\documentclass[11pt,reqno]{amsart}

\usepackage[english]{babel}
\usepackage[T1]{fontenc}
\usepackage[utf8]{inputenc}
\usepackage{csquotes}

\usepackage[title,titletoc,page]{appendix}
\usepackage[a4paper,left=3cm,right=3cm,top=2cm,bottom=4cm,bindingoffset=5mm]{geometry}
\usepackage[colorlinks=true,linkcolor=blue,linktocpage=true, citecolor=black]{hyperref}
\usepackage[all]{hypcap}
\usepackage{setspace}
\usepackage{xcolor}

\usepackage{aligned-overset}
\usepackage{amsfonts}
\usepackage{amsmath}
\usepackage{amssymb}
\usepackage{amsthm}
\usepackage{dsfont}
\usepackage{bm}
\usepackage{float}
\usepackage{mathrsfs}
\usepackage{mathtools}
\mathtoolsset{showonlyrefs}

\usepackage{thmtools}
\usepackage[shortlabels]{enumitem}

\allowdisplaybreaks

\usepackage{graphicx}

\usepackage{mathscinet}
\usepackage[backend=biber,sorting=nty,natbib=true,isbn=false,doi=false,url=false,maxbibnames=5,style=ieee]{biblatex}
\usepackage{version}
\usepackage[obeyFinal]{todonotes}
\setuptodonotes{inline}
\usepackage{cancel}
\DeclareMathOperator{\R}{\mathbb{R}}

\DeclareMathOperator{\eps}{\varepsilon}

\numberwithin{equation}{section}

\newtheorem{theorem}{Theorem}[section]
\newtheorem{lemma}[theorem]{Lemma}
\newtheorem{corollary}[theorem]{Corollary}
\newtheorem{proposition}[theorem]{Proposition}

\theoremstyle{definition}
\newtheorem{definition}[theorem]{Definition}

\newtheorem{remark}[theorem]{Remark}
\newtheorem{assumption}[theorem]{Assumption}

\newtheorem{example}[theorem]{Example}

\usepackage{bbm}
\usepackage[normalem]{ulem}
\usepackage{cancel}

\newcommand{\cB}{\mathcal{B}}

\newcommand{\cD}{\mathcal{D}}

\newcommand{\cH}{\mathcal{H}}
\newcommand{\cI}{\mathcal{I}}

\newcommand{\cM}{\mathcal{M}}
\newcommand{\cN}{\mathcal{N}}

\newcommand{\cV}{\mathcal{V}}

\newcommand{\bA}{\mathbb{A}}
\newcommand{\bB}{\mathbb{B}}

\newcommand{\bF}{\mathbb{F}}
\newcommand{\bG}{\mathbb{G}}
\newcommand{\bH}{\mathbb{H}}

\newcommand{\bN}{\mathbb{N}}

\newcommand{\bR}{\mathbb{R}}

\newcommand{\bfV}{\mathbf{V}}

\newcommand{\bfd}{\mathbf{d}}

\newcommand{\bff}{\mathbf{f}}

\newcommand{\bfz}{\mathbf{z}}

\newcommand{\bfzeta}{\bm{\zeta}}
\newcommand{\bfxi}{\bm{\xi}}

\newcommand{\dd}{ \mathrm{d}}
\def\eps{\varepsilon}

\DeclareMathOperator{\Tr}{Tr}

\DeclareRobustCommand{\rchi}{{\mathpalette\irchi\relax}}
\newcommand{\irchi}[2]{\raisebox{\depth}{$#1\chi$}} 

\newcommand{\vn}[1]{\left| \! \left| #1\right| \!\right|}

\newcommand{\ip}[2]{\left\langle #1,#2\right\rangle}

\newcommand{\ssup}[1]{\left\lceil #1 \right\rceil}
\newcommand{\iinf}[1]{\left\lfloor #1 \right\rfloor}
\newcommand{\interior}[1]{\mathrm{int}(#1)}
\renewcommand{\interior}[1]{\mathring{#1}}

\newcommand{\domain}{E}

\title[A Strict Comparison Principle on Domains with Boundary]{A Strict Comparison Principle for Integro-Differential Hamilton--Jacobi--Bellman Equations on Domains with Boundary}

\author{Serena Della Corte}
\address{Mathematical Institute, Leiden University, The Netherlands}
\email{s.della.corte@math.leidenuniv.nl}

\author{Fabian Fuchs}
\address{Dipartimento di AI, Data and Decision Sciences, LUISS University, Roma, Italy}
\email{ffuchs@luiss.it}

\author{Richard C. Kraaij}
\address{Delft Institute of Applied Mathematics, Delft University of Technology, The Netherlands}
\email{r.c.kraaij@tudelft.nl}

\author{Max Nendel}
\address{Department of Statistics and Actuarial Science, University of Waterloo, Canada}
\email{mnendel@uwaterloo.ca}

\thanks{The second and fourth named authors are funded by the Deutsche Forschungsgemeinschaft (DFG, German Research Foundation) -- SFB 1283/2 2021 -- 317210226.\newline
The first and third named authors are supported by The Netherlands Organisation for Scientific Research (NWO), grant numbers 613.009.148 and VI.Vidi.233.092}

\date{\today}

\begin{document}

\begin{abstract}
    This work provides a comparison principle for viscosity solutions to boundary value problems on (partially) bounded, cylindrical spaces.\ The comparison principle is based on a test function framework that allows for the simultaneous treatment of diffusive as well as jump terms.
    Estimates in the proof of the comparison principle incorporate the use of Lyapunov functions that act as growth bounds for the solutions, effectively yielding a theory for unbounded viscosity solutions.
    We apply the results to a wide class of parabolic equations and elliptic problems on a space with corners.\smallskip

    \smallskip
    
\end{abstract}

\maketitle


\section{Introduction}
In this work, we extend the framework from the authors' previous work \cite{DCFuKrNe24} and provide a strict comparison principle for viscosity solutions to boundary value problems of the type
\begin{equation}\label{into:eq:bvp}
    \begin{cases}
        Hf(x)=0, & x \in E,\\
        Gf(x)=0, & x \in \partial E,
    \end{cases}
\end{equation}
on a (partially) bounded, cylindrical space $E$ with a boundary satisfying the exterior sphere condition.\ This setup explicitly includes parabolic equations and elliptic equations on spaces with corners, including, e.g., finite time horizon optimal control problems, see \cite{FlSo06,YoZh99,FaGoSw17} for an overview. Other extensions in this paper are related to the treatment of unbounded solutions with growth in terms of a Lyapunov function and the introduction of the choice of doubling penalization.  

Our results recover the classical results on oblique derivative problems for second-order partial differential equations on a quadrant, see \cite{DuIs90,DuIs91}, as well as more recent generalizations, cf.\ \cite{IsKu22,BiIsSaWa17,GuMoSw19,BiBrIs24,BiBrIs24a}, and extend them to include jump terms, appearing, for example, in the context of  model uncertainty for Lévy processes, cf.\ \cite{hu.2021,Pe07,neufeld.2017,zbMATH07173248}.

Viscosity solutions were first introduced in the seminal paper \cite{CrLi83}, where Hamilton--Jacobi equations, but especially Dirichlet and Cauchy problems on bounded domains, were among the first motivating examples.\
Since then, both elliptic and parabolic equations of first and second order have been extensively treated in the literature, for foundational results see \cite{CrNe85,CrIsLi87,Is86,Is84,Li83,BaPe90} and for more specialized results, for non-standard spaces and boundary conditions, see \cite{Ba99,Is91, DuIs90,DuIs91,Ber2023}. 

We refer to the \emph{User's Guide}, cf.\ \cite{CIL92}, for an overview of uniqueness results for problems including second-order terms as well as the needed adjustments of the well-known \emph{Crandall--Ishii Lemma} when moving from elliptic to parabolic equations and from unbounded to bounded spaces. Integral operators related to jump processes have been treated in \cite{JaKa05,JaKa06,BaIm08,GuMoSw19,hu.2021}. \

The starting point of this paper is the framework developed in \cite{DCFuKrNe24}. There, the Crandall--Ishii Lemma is reformulated in terms of test functions and the main estimate on the difference of Hamiltonians is treated using couplings of operators. This provides a unified approach to show a strict comparison principle for second-order and non-local equations.
The framework in \cite{DCFuKrNe24}, however, only concerns equations on unbounded spaces and is formulated for bounded viscosity solutions. 

In this work, we extend the approach to (partially) bounded cylindrical spaces, possibly with corners, and to unbounded viscosity solutions.
The main result of this work, Theorem \ref{th:comparison_HJI}, provides a strict comparison principle for boundary value problems of form \eqref{into:eq:bvp}, where, in the interior, we consider Hamiltonians $H$ of the type
\begin{multline} \label{eqn:operator_linear_intro}
    H f(x) =\ip{b(x)}{\nabla f (x)} +  \frac{1}{2} \Tr\left(\Sigma\Sigma^T (x) D^2 f(x) \right) \\ 
    + \int \left[f(x+\bfz) - f(x) - \rchi_{B_1(0)} (\bfz) \ip{\bfz}{\nabla f(x)}\right] \mu_x(\dd \bfz) + \cH(\nabla f(x))
\end{multline}
as well as their Hamilton--Jacobi--Bellman (HJB) version
\begin{align}
    Hf(x) &= \sup_{\theta \in \Theta} \{H_{\theta}f(x)-\cI(x,\theta)\}, 
\end{align}
where $H_{\theta}$ are Hamiltonians as in \eqref{eqn:operator_linear_intro} but with $\theta$ dependent coefficients and $\cI$ is an appropriate cost function.

This extension requires more than a direct application of the results in \cite{DCFuKrNe24}.

In fact, a more detailed analysis is needed to 
account for the geometry of the underlying space, the viscosity formulation of boundary conditions, and the growth of the solutions simultaneously. In particular, the geometry and viscosity formulation influence the form of the boundary operator $G$, see the discussion at the beginning of Section \ref{sec:main_thm}.

Consequently, the arguments in \cite{DCFuKrNe24} need to be refined and adapted to the new setting. To this end, we make systematic use of a Lyapunov function, which plays a central role in our analysis.
First, in Propositions \ref{prop:optimizer_construction} and \ref{proposition:test_function_construction}, the Lyapunov function provides a natural growth bound for the solutions and localizes the optimizing points in explicitly recoverable compact sets.
Second, in Proposition \ref{proposition:basic_comparison_using_Hestimate}, which is the main ingredient in the proof of our main result, Theorem \ref{th:comparison_HJI}, it allows us to perform a reduction-to-the-boundary argument providing more precise information on the location of the optimizers and their limiting points relative to the boundary, which is an essential component to conclude with the strict comparison principle. 
The role of the Lyapunov function in the reduction step is thus, in spirit, similar to the literature on the so-called \emph{twin blow-up method}, cf.\ \cite{FoImMo24,FoImMo25,LiSo17}.

In practice, the combination of coefficients of the boundary value problem that we can treat depends on the existence of an appropriate Lyapunov function, and thus indirectly on the geometry of the space as well, cf.\ Remark \ref{remark:Lyapunov_coeffs}.

Though we adapt the framework first presented in \cite{DCFuKrNe24}, the treatment of unbounded solutions and choice of doubling penalizations, but most importantly the boundedness (and potential irregularity) of the underlying space, introduce additional difficulties in the constructions of optimizers and test functions in Section \ref{sec:opt_construction} as well as the proof of the strict comparison principle in Section \ref{section:main_proof}, which are overcome in this work.

The rest of the paper is organized as follows: Section \ref{sec:framework} contains the basic definitions. Section \ref{sec:setup_mainthm} sets up the framework and states the main result as well as necessary assumptions. In Section \ref{sec:applications}, we apply the main result first to parabolic equations and then to elliptic equations on the quadrant.\ The modified optimizer and test function construction is contained in Section \ref{sec:opt_construction}, while the proof of the main theorem can be found in Section \ref{section:main_proof}.

\section{Preliminaries and general setting}\label{sec:framework}
\subsection{Notation and Preliminaries}\label{sec:framework_notations}

Throughout the paper, for $q_1, q_2 \in \bN$, we consider a closed and bounded set $E_1 \subseteq \bR^{q_1}$ and a closed set $E_2 \subseteq \bR^{q_2}$, and denote $E=E_1 \times E_2$. We assume that, for $i\in\{1,2\}$, the interior of $E_i$ is a domain, i.e., non-empty and connected. We endow $E$ with the product topology and metric $\bfd^2 = d_1^2+d_2^2$, where, for $i\in\{1,2\}$, $d_i$ is the Euclidian distance on $\bR^{q_i}$. Furthermore, we denote the boundary of $E$ as $\partial E$ and its interior as $\interior{E}$. We denote the set of outward normals at $x \in \partial E$ as $N_x$. Furthermore, we impose the exterior sphere condition on $E$, i.e., there exists a $r_E > 0$ such that, for all $x \in \partial E$ and outward normal vectors $n(x) \in N_x$, we have
\begin{equation}
    B_{r_E}(x+rn(x)) \cap E = \emptyset.
\end{equation}

Denote the set of upper and lower semicontinuous functions as $\mathrm{USC}(E)$ and $\mathrm{LSC}(E)$, respectively. Let $C(E)$ and $C_b(E)$ be the set of continuous and bounded continuous functions. For $k \in \bN$, let $C^k(E)$ denote the space of all real-valued functions on $E$ that are $k$-times continuously differentiable. In particular, derivatives on the boundary of $E$ are understood in the classical sense as the restriction of an extension of functions to an open neighborhood of $E$, denoted as $\cN (E)$.
Let $C^k_b(E)$ the set of all functions in $C^k(E)$ with bounded derivatives up to order $k\in \bN$.\ 
We write $C_u(E)$ and $C_l(E)$ for the set of continuous functions on $E$ that are bounded from above and below, respectively. 
Moreover, we write
\begin{align*}
	C_+(E) & := \{f \in C(E) \, | \, f \text{ has compact sub-level sets}\}, \\
	C_-(E) & := \{f \in C(E) \, | \, f \text{ has compact super-level sets}\}, \\
	C_c(E) & := \{f \in C(E) \, | \, f \text{ is constant outside of a compact set}\}.
\end{align*}
We furthermore define the following intersections: $C_c^2(E) = C_c(E) \cap C^2(E)$,
\begin{equation*}
	C_+^2(E) \coloneqq C_+(E) \cap C^2(E), \qquad C_-^2(E) \coloneqq C_-(E) \cap C^2(E).
\end{equation*}
Moreover, $C_c^\infty(E)$ denotes the space of all smooth functions that are constant outside of a compact set. 

For $a,b \in \bR$, we write $a \vee b \coloneqq \max \{a,b\}$ and $a \wedge b \coloneqq \min \{a,b\}$.
We denote the supremum norm by $\vn{\,\cdot\,}$, that is
\begin{equation}
    \vn{f} = \sup_{x\in E} |f(x)|,
\end{equation}
for $f \in C_b(E)$, while, for $u \in C(E)$, we use the notation
\begin{equation}
\ssup{u} \coloneqq \sup_{x \in E} u(x) \quad\text{or}\quad \iinf{u} \coloneqq \inf_{x \in E} u(x)
\end{equation}
for a supremum or infimum over the entire space and 
\begin{equation*}
   \ssup{u}_C \coloneqq \sup_{x \in C} u(x) \quad\text{or}\quad \iinf{u}_C \coloneqq \inf_{x \in C} u(x)
\end{equation*}
for a supremum or infimum over a subset $C \subseteq E$, respectively.

For two functions $f,g \colon E \to \bR$, we say that $f \in o(g)$, if, for all $\delta >0$, there exists a radius $r>0$ such that, for all $x\in E$ with $|x|>r$, we have $|f(x)|\leq \delta|g(x)|$.

As we make frequent use of them, we distinguish the direct sum on $E=E_1 \times E_2$ from the direct sum on $E \times E$: 
For functions $f_1\colon E_1 \rightarrow \bR$ and $f_2\colon E_2 \rightarrow \bR$, a vector of constants $C=(C_1, C_2) \in \bR^2$, and $x=(x_1,x_2) \in E_1\times E_2$, we denote the \emph{direct sum on $E$} as 
\begin{equation}
    C \bff(x) \coloneqq C_1f_1(x_1) + C_2f_2(x_2).
\end{equation}
If $C \in \bR$, we identify $C= (C,C)$ to define $C \bff(x)$.
For the \emph{direct sum on $E\times E$}, we use the usual notation and, for $f_1, f_2 \in C(E)$, write $f_1 \oplus f_2, f_1 \ominus f_2 \in C(E\times E)$ to mean
\begin{equation*}
    (f_1 \oplus f_2)(x, x') \coloneqq f_1(x) + f_2(x') \quad\text{and}\quad (f_1 \ominus f_2)(x, x') \coloneqq f_1(x) - f_2(x')
\end{equation*}
for all $x,x'\in E$. Furthermore, for two sets of functions $F_1, F_2 \subseteq C(E)$, we define
\begin{equation*}
    F_1 \oplus F_2 \coloneqq \left\{f_1 \oplus f_2 \,\middle|\, f_1 \in F_1, f_2 \in F_2\right\} \quad\text{and}\quad
    F_1 \ominus F_2 \coloneqq \left\{f_1 \ominus f_2 \,\middle|\, f_1 \in F_1, f_2 \in F_2\right\}.
\end{equation*}

For a vector of constants $C = (C_1, C_2) \in \bR^2$, we write
\begin{equation}
    1\boxplus C \coloneqq (1+C_1+C_2),\quad 1\boxminus C \coloneqq (1-C_1-C_2).
\end{equation}

We say that a function $\omega \colon [0,\infty) \rightarrow [0,\infty)$ is a \emph{modulus of continuity}, if $\omega$ is upper semi-continuous with $\omega(0) = 0$.
We say that a function $f$ \emph{admits a modulus of continuity} $\omega_K \colon [0,\infty) \to [0,\infty)$ on a compact set $K \subseteq E$ if, for all $x, y \in K$, we have
\begin{equation}\label{eq:modulus_admission}
    | f(x) - f(y) | \leq \omega_K(\bfd(x,y)).
\end{equation}

A function $\phi \colon E \rightarrow \bR$ is called \emph{semi-convex} with constant $\kappa \in (0, \infty)$ if, for some $x_0 \in E$, the map
\begin{equation*}
    x \mapsto \phi(x) + \frac{\kappa}{2}\bfd^2(x,x_0)
\end{equation*}
is convex along any segment contained in $E$, i.e., for any $x, y \in E$, the map is convex along the segment $\{ \lambda x + (1-\lambda)y \,|\, \lambda \in [0,1]\}$ if it is contained in $E$. Moreover, $\phi$ is called \emph{semi-concave} with constant $\kappa \in (0, \infty)$ if $-\phi$ is semi-convex with constant $\kappa$.

We say that a function $f \in C(E,\bR^{q_1+q_2})$ is \emph{one-sided Lipschitz} if there is some constant $C \in \bR$ such that, for all $x,y \in E$, 
\begin{equation}\label{def:OneSideLip}
    \ip{x-y}{f(x) - f(y)} \leq C \bfd^2(x,y).
\end{equation}

For any $z \in E$, let $s_z : E \rightarrow \bR^{q_1+q_2}$ be the \emph{shift map}
\begin{equation}\label{definition:shift_map}
    s_z(x) = x-z.
\end{equation}

For any $z_1, z_2 \in E$ and function $\Phi\colon E^2 \rightarrow \bR$, let
\begin{equation} \label{eqn:shiftPhi}
        \Phi_{z_1, z_2} (x,y) \coloneqq \Phi\left(s_{z_1} (x), s_{z_2} (y)\right).
\end{equation}

\subsection{Operators}

We consider operators of type $H \subseteq C(E) \times \big(C(\interior{E}) \cap \mathrm{USC}(E)\big)$ and $H \subseteq C(E) \times \big(C(\interior{E}) \cap \mathrm{LSC}(E)\big)$, where we identify $H$ by its graph. 
As usual, the \emph{domain} of $H$ is given by
\begin{equation*}
    \cD(H) \coloneqq \left\{f\in C(E)\,\middle|\, \exists \,g\in \mathrm{Im}(H) \colon (f,g)\in H\right\},
\end{equation*}
where $\mathrm{Im}(H)$ is the co-domain of $H$.
Throughout, all operators are assumed to be single-valued. Thus, for every $f\in\cD(H)$, there exists a unique $g$ such that $(f,g)\in H$, and we write $Hf=g$.
Let $H_1, H_2 \subseteq C(E) \times \big(C(\interior{E}) \cap \mathrm{USC}(E)\big)$ or $H_1, H_2 \subseteq C(E) \times \big(C(\interior{E}) \cap \mathrm{LSC}(E)\big)$. We define
\begin{equation*}
    H_1 + H_2 \coloneqq \left\{(f,g_1+g_2)\, \middle|\, (f,g_1) \in H_1, (f,g_2) \in H_2 \right\},
\end{equation*}
which is an operator with domain
\begin{equation*}
    \cD(H_1 + H_2) \coloneqq \cD(H_1) \cap \cD(H_2).
\end{equation*}
We say that $H$ is \emph{linear on its domain} if, for any $f_1,f_2 \in \cD(H)$ and $a \in \bR$ such that $af_1 + f_2 \in \cD(H)$, we have
\begin{equation}\label{definition:extended_linear}
    H \left(af_1+ f_2\right) = aHf_1 + Hf_2.
\end{equation}

\subsection{Viscosity solutions}

For operators $F_1 \subseteq C_l(E) \times C(E)$, $G_1 \subseteq C_l(E) \times C(\partial E)$ and $F_2 \subseteq C_u(E) \times C(E)$, $G_2 \subseteq C_u( E) \times C(\partial E)$, we consider the boundary value problems
\begin{equation}
    H_1f(x)=\begin{cases}
        F_1f(x), &\text{if } x \in E,\\
        G_1f(x), &\text{if } x \in \partial E,
    \end{cases} 
    \quad\text{and}\quad
    H_2f(x)=\begin{cases}
        F_2f(x), &\text{if } x \in E,\\
        G_2f(x), &\text{if } x \in \partial E,
    \end{cases} 
\end{equation}
and study the equations
\begin{align}
    H_1 f & \leq 0, \label{eqn:HJ_subsolution} \\
    H_2 f & \geq 0. \label{eqn:HJ_supersolution}
\end{align}

The notion of viscosity solution is built upon the maximum principle.
\begin{definition}[Maximum principle]\label{def:max_principle}
    We say that operators $H_1 \subseteq C(E) \times \big(C(\interior{E}) \cap \mathrm{USC}(E)\big)$ and $H_2 \subseteq C(E) \times \big(C(\interior{E}) \cap \mathrm{LSC}(E)\big)$ satisfy the \emph{maximum principle} if, for all $f_1,f_2 \in \cD(H_1)$ and $x_0 \in E$ with
    \begin{equation*}
        f_1(x_0) - f_2(x_0) = \sup_{x \in E} \{f_1(x) - f_2(x)\},
    \end{equation*}
    we have
    \begin{equation*}
        H_1 f_1(x_0) \leq H_1 f_2(x_0)
    \end{equation*}
    and, analogously, for all $f_1,f_2 \in \cD(H_2)$ and $x_0 \in E$ with
    \begin{equation*}
        f_1(x_0) - f_2(x_0) = \inf_{x \in E} \{f_1(x) - f_2(x)\},
    \end{equation*}
    we have
    \begin{equation*}
        H_2 f_1(x_0) \geq H_2 f_2(x_0).
    \end{equation*}
\end{definition}

As we want to treat unbounded solutions, we have to specify a growth bound for the solutions. We phrase this growth in terms of a function $\cV$ with compact sublevel sets. Later, the containment function, cf.\ Definition \ref{definition:perturbation_containment}, will play the role of $\cV$. In the definition of viscosity solution below, intuitively, one takes $f \approx \mathfrak{f} +\cV$ in the definition of subsolutions and $f \approx \mathfrak{f} - \cV$ in the definition of supersolutions, where $\mathfrak{f}$ is some bounded smooth function, e.g., $\mathfrak{f} \in C_c^\infty (E)$.

\begin{definition}[Viscosity sub- and supersolutions] \label{definition:viscosity_solutions} \label{def:viscosity_solution}
    Let $F_1 \subseteq C_l(E) \times C(E)$, $G_1 \subseteq C_l(E) \times C(\partial E)$ and $F_2 \subseteq C_u(E) \times C(E)$, $G_2 \subseteq C_u(E) \times C(\partial E)$ be operators with domains $\cD(F_1)$, $\cD(G_1)$, $\cD(F_2)$, and $\cD(G_2)$, respectively. Then, for $f_1 \in \cD(F_1) \cap \cD(G_1)$, let
    \begin{equation}\label{eq:bvp_vis_subsolution}
        H_1 f_1(x) = 
        \begin{cases}
            F_1f_1(x), &\text{if } x \in E,\\
            \min\big\{F_1 f_1(x),\,G_1f_1(x)\big\}, &\text{if } x \in \partial E,
        \end{cases}
    \end{equation}
    be an operator with domain $\cD(H_1) = \cD(F_1) \cap \cD(G_1)$ and, for $f_2 \in \cD(F_2) \cap \cD(G_2)$, let 
    \begin{equation}\label{eq:bvp_vis_supersolution}
        H_2 f_2(x) = 
        \begin{cases}
            F_2f_2(x), &\text{if } x \in E,\\
            \max\big\{F_2 f_2(x),\,G_2f_2(x)\big\}, &\text{if } x \in \partial E
        \end{cases}
    \end{equation}
    be an operator with domain $\cD(H_2) = \cD(F_2) \cap \cD(G_2)$.

    \begin{enumerate}[(a)]
        \item An upper semicontinuous function $u \colon E\to \R$ is called a \emph{(viscosity) subsolution} to \eqref{eqn:HJ_subsolution} if, for all $(f,g) \in H_1$ with $f \in \cD(H_1) \cap C_+(E)$, there exists a sequence $(x_n)_{n\in \bN} \subseteq E$ such that
        \begin{gather*}
            \lim_{n \rightarrow \infty} u(x_n) - f(x_n)  = \sup_{x\in E} u(x) - f(x), \\
            \limsup_{n \rightarrow \infty} g(x_n) \leq 0.
        \end{gather*}
        
        \item A lower semicontinuous function $v \colon E\to \R$ is called a \emph{(viscosity) supersolution} to \eqref{eqn:HJ_supersolution} if, for all $(f,g) \in H_2$ with $f \in \cD(H_2) \cap C_-(E)$, there exists a sequence $(x_n)_{n\in \bN} \subseteq E$ such that
        \begin{gather*}
            \lim_{n \rightarrow \infty} v(x_n) - f(x_n) = \inf_{x\in E} v(x) - f(x), \\
            \liminf_{n \rightarrow \infty} g(x_n) \geq 0.
        \end{gather*}
    \end{enumerate}

    A function $u \in C(E)$ is called a \emph{(viscosity) solution} if it is both a subsolution to \eqref{eq:bvp_vis_subsolution} and a supersolution to \eqref{eq:bvp_vis_supersolution}.
\end{definition}

Working with test functions that have compact sub- or superlevel sets respectively, an approximating sequence can be replaced by an optimizing point in the definition of sub- or supersolution. See \cite[Lemma D.1]{DCFuKrNe24}.

\begin{remark}
    Note that the subsolution $u$ and supersolution $v$ as well as their respective test functions $f$ in the above definition might be unbounded. For a setting with bounded sub- and supersolutions, but especially bounded test functions, we refer to the setting in \cite{DCFuKrNe24} with sequential denseness.
\end{remark}

Associated with the definition of viscosity solutions, we introduce the comparison principle on spaces with boundary, which implies uniqueness in the viscosity sense for solutions to equations of the form $H f = 0$. We additionally introduce the strict comparison principle, which is a stronger notion. For a discussion on the strict comparison principle, we refer to \cite{DCFuKrNe24}.

\begin{definition} \label{definition:comparison}
    We say that the equations \eqref{eqn:HJ_subsolution} and \eqref{eqn:HJ_supersolution} satisfy 
    
    \begin{enumerate}[(a)]
        \item the \emph{comparison principle} if, for any subsolution $u$ to \eqref{eqn:HJ_subsolution} and any supersolution $v$ to \eqref{eqn:HJ_supersolution}, we have 
    \begin{equation*}
        \sup_{x\in E} u(x) - v(x) \leq \sup_{x \in \partial E} u(x) - v(x).
    \end{equation*}
    \item  the \emph{strict comparison principle} if, for any subsolution $u$ to \eqref{eqn:HJ_subsolution}, any supersolution $v$ to \eqref{eqn:HJ_supersolution}, any compact set $K\subseteq E$ and $\eps = (\eps_1, \eps_2)$ with $\eps_1, \eps_2 >0$, there exist a compact set $\widetilde{K} = \widetilde{K}(K,\varepsilon, u, v) \subseteq \partial E$ and a constant $C = C(u,v)$ such that we have 
    \begin{equation*}
        \sup_{x\in K} u(x) - v(x)  \leq  \varepsilon C + \sup_{x \in \widetilde{K}} u(x) - v(x).
    \end{equation*}
    \end{enumerate}
\end{definition}

Note that also in the setting with boundary, the strict comparison principle implies the comparison principle: The strict comparison principle implies that, for any $x_0 \in E$ and $\eps = (\eps_1, \eps_2)$ with $\eps_1, \eps_2 >0$, there exists a constants $C = (C_1, C_2)$, independent of $\eps$, and a compact set $\widetilde{K} \subseteq \partial E$ such that
\begin{equation*}
    u(x_0) - v(x_0)  \leq  \varepsilon C + \sup_{x \in \widetilde{K}} u(x) - v(x) \leq \varepsilon C + \sup_{x \in \partial E} u(x) - v(x).
\end{equation*}
Now, letting $\eps_1, \eps_2 \downarrow 0$ and taking the supremum over all $x_0 \in E$, the comparison principle follows.
\section{Setup and Main Results}\label{sec:setup_mainthm}
In this section, we adapt the framework in \cite{DCFuKrNe24} to the setting of (partially) bounded spaces. Our main result, Theorem \ref{th:comparison_HJI}, is a strict comparison principle for equations of the type
\begin{equation}\label{eq:setup:main_operator}
    \bH f(x) = \sup_{\theta \in \Theta} \big\{ \bH_\theta f(x) - \cI(x,\theta) \big\} = \sup_{\theta \in \Theta} \big\{\bA_{\theta} f(x) + \bB_{\theta} f(x) - \cI(x,\theta) \big\}
\end{equation}
in Hamilton--Jacobi--Bellman (HJB) form.

Section \ref{sec:notion_our_framework} introduces the framework and its core concepts.\ Section \ref{sec:main_thm} states the strict comparison principle, while the necessary technical assumptions are deferred to Section \ref{subsection:regularity_and_compatibility}.

\subsection{Concepts of the framework}\label{sec:notion_our_framework} 

In this subsection, we give the main definitions underlying our framework. We consider operators of the form 
\[H = A+B,\]
where $A$ is a stochastic part, which we can couple in the sense of the Definition \ref{def:coupling} below, and $B$ is a deterministic part, which we require to be a \textit{convex semi-monotone} operator in the sense of Definition \ref{definition:first_order} below.\ Later, $H$ will play the role of $\bH_\theta$ in equation \eqref{eq:setup:main_operator}.

\begin{definition}[Coupling]\label{def:coupling:only_coupling}
    Let $A \subseteq C(E) \times C(E)$ and $\widehat{A} \subseteq C(E^2) \times C(E^2)$ be linear on their respective domains. We say $\widehat{A}$ is a \emph{coupling of $A$} if  $\cD(A) \oplus \cD(A) \subseteq \cD(\widehat{A})$ and, for any $f_1, f_2 \in \cD(A)$, we have
    \begin{equation*}
        \widehat{A} \left(f_1 \oplus f_2\right) = A f_1 + A f_2.
    \end{equation*}
\end{definition}

\noindent Recall the notation $\Phi_{z,z'}$ introduced in \eqref{eqn:shiftPhi}. 

\begin{definition}[$\Phi$-controlled growth]\label{def:coupling:growth}
    Let $\widehat{A} \subseteq C(E^2) \times C(E^2)$. For some function $\Phi \colon E^2 \rightarrow \bR$, we say that $\widehat{A}$ has \emph{$\Phi$-controlled growth} if, for any $\alpha = (\alpha_1, \alpha_2)$ with $\alpha_1, \alpha_2 > 1$ and $z,z'\in E$, we have $\alpha\Phi_{z,z'} \in \cD(\widehat{A})$.
    
    In addition, for any compact set $K \subseteq E$, there exists a modulus of continuity $\omega_{\widehat{A}, K}\colon [0, \infty) \rightarrow [0, \infty)$ and $x,x',y,y' \in K$ such that
    \begin{multline}
        \widehat{A}\left(\alpha\Phi_{x-y,\ x'-y'}\right)(x,x') \leq \omega_{\widehat{A}, K}\Big(\alpha \big( (\bfd(x,y) + \bfd(y', x'))^2 + \Phi(y,y')\big) \\
        + \left(\bfd(x,y) + \bfd(y', x') + \Phi(y,y') \right)\Big).
    \end{multline}
\end{definition}

\begin{definition}[$\Phi$-controlled growth coupling]\label{def:coupling}
    Let $A \subseteq C(E) \times C(E)$ and $\widehat{A} \subseteq C(E^2) \times C(E^2)$ be linear on their respective domains. We say $\widehat{A}$ is a \emph{$\Phi$-controlled growth coupling of $A$} if the following properties are satisfied:
    \begin{enumerate}[(a)]
        \item $\widehat{A}$ satisfies the \emph{maximum principle}, cf. Definition \ref{def:max_principle}.
        \item $\widehat{A}$ is a \emph{coupling} of $A$, cf. Definition \ref{def:coupling:only_coupling}.
        \item $\widehat{A}$ has \emph{$\Phi$-controlled growth}, cf. Definition \ref{def:coupling:growth}.
    \end{enumerate}
\end{definition}

\begin{definition}[Local first-order operator]\label{def:first_order:only_first_order}
    We say that $B \subseteq C(E) \times C(E)$ is a \emph{local first-order} operator if there exists a continuous map $\cB : E \times \bR^{q_1+ q_2} \rightarrow \bR$ such that, for any $f \in \cD(B)$, we have $Bf(x) = \cB(x,\nabla f(x))$.
\end{definition}

\begin{definition}[Local semi-monotonicity w.r.t.\ $\Phi$]\label{def:oneSidedLipschitz}
    Let $B \subseteq C(E) \times C(E)$ be local first-order for some $\cB$, cf. Definition \ref{def:first_order:only_first_order}.
    For some function $\Phi \colon E^2 \rightarrow \bR$, we say that $B$ is \emph{locally semi-monotone w.r.t.\ $\Phi$} if, for any compact set $K \subseteq E$, there exists a modulus of continuity $\omega_{\cB,K}\colon [0,\infty)\to [0,\infty)$ such that, for all $x,x' \in K$ and $\alpha = (\alpha_1, \alpha_2)$ with $\alpha_1, \alpha_2 > 1$, we have
    \begin{equation*}
        \cB(x,\alpha(x-x')) - \cB(x',\alpha(x-x')) \leq \omega_{\cB,K}\left(\alpha \big(\bfd^2(x,x')\wedge \Phi(x,x')\big) + \bfd(x,x') + \Phi(x,x') \right).
    \end{equation*}
\end{definition}

\begin{definition}[Convex semi-monotone operator w.r.t.\ $\Phi$] \label{definition:first_order}
    We say that $B \subseteq C(E) \times C(E)$ is a \emph{convex semi-monotone operator w.r.t.\ $\Phi$} if the following properties are satisfied:
    \begin{enumerate}[(a)]
        \item $B \subseteq C(E) \times C(E)$ is \emph{locally semi-monotone w.r.t.\ $\Phi$} for some $\cB$, cf. Definition \ref{def:oneSidedLipschitz}.
        \item For all $x\in E$, the map $p \mapsto \cB(x, p)$ is \emph{convex}.
    \end{enumerate}
\end{definition}

As before in \cite{DCFuKrNe24}, the proof of the comparison principle will use a perturbed doubling-of-variables type argument on the optimization problem $\sup (u-v)$.\ 
To that end, we need a doubling penalization and the following perturbation functions:
\begin{itemize}
    \item We need a function $\bfV$ that allows us to (1) work with compact sets and (2) serves as a growth bound if $E_1$ or $E_2$ is unbounded, see Definition \ref{definition:perturbation_containment}.
    \item Since we want to construct optimizers using a Jensen-type result, cf.\ Proposition \ref{proposition:Jensen_Alexandrov_cutoff}, we need families of locally linear perturbation $\bfzeta_{p, z}$ and localization $\bfxi_z$, cf.\ Definition \ref{definition:perturbation_first_second_order}.
\end{itemize}

\begin{definition}[Doubling penalization]\label{def:doubling_pen}
    For $i \in \{1,2\}$, we say that a function $\phi_i \colon E_i \times E_i \rightarrow [0, \infty)$ is a \emph{doubling penalization for $E_i$} if
    \begin{enumerate}[(a)]
        \item $\phi_i$ is lower semi-continuous, semi-concave with semi-concavity constant $\kappa_{\phi_i}$, and can be extended to a neighborhood $\cN (E_i \times E_i)$ such that it is semi-concave with the same constant.
        \item $\phi_i$ separates points, i.e., for all $x_i, y_i \in E_i$ with $x_i \neq y_i$, we have $\phi_i(x_i,x_i) = 0$ and $\phi_i(x_i,y_i) > 0$.
    \end{enumerate}
    We say that a function $\Phi \colon E \times E \rightarrow [0, \infty)$ is a \emph{doubling penalization} if
    \begin{equation}
        \Phi(x,y) = \phi_1(x_1, y_1) + \phi_2(x_2, y_2),
    \end{equation}
    where $x,y \in E$ and $\phi_1, \phi_2$ are doubling penalization for $E_1$ and $E_2$.
\end{definition}

\begin{example}[Doubling penalization]
    The prototypical choice of doubling penalization is $\phi_i (x,x') = \frac{1}{2}d^2_i(x,x')$ for $i \in \{1,2\}$.
\end{example}

\begin{definition} \label{definition:perturbation_containment}
    For $i \in \{1,2\}$, we call $V_i : E_i \rightarrow \bR$ a \emph{containment function for $E_i$} if we have
    \begin{enumerate}[(a)]
        \item\label{item:definition:perturbation_containment:bounded_below} $V_i$ is bounded from below,
        \item\label{item:definition:perturbation_containment:semi_concave} $V_i$ is semi-concave with semi-concavity constant $\kappa_{V_i}$ and can be extended to a neighborhood $\cN(E_i)$ such that it is semi-concave with the same constant,
        \item\label{item:definition:perturbation_containment:compact_sublevels} for every $c \in \bR$, the set $\{y \, | \, V_i(y) \leq c\}$ is compact. 
    \end{enumerate}

    We the call $\bfV \colon E \rightarrow \bR$ a \emph{containment function} if
    \begin{equation}
        \bfV(x) = V_1(x_1) + V_2(x_2),
    \end{equation}
    where $x \in E$ and $V_1, V_2$ are containment functions for $E_1$ and $E_2$.
\end{definition}

Depending on the operator and the underlying space, typical examples of containment functions in the literature are $V_i (x) \approx \lvert x \rvert$ or $V_i(x) \approx \lvert x \rvert^2$ but also $V_i(x) \approx \log(1+\vert x \vert^2)$ for $i \in \{1,2\}$. Later, in Section \ref{sec:applications}, we will use the following containment functions
\begin{example}[Containment functions]
    Let $i \in \{1,2\}$. Depending on the operator, if $\partial E_i \neq \emptyset$, we consider
    \begin{align}
        V_i(x) = x \quad\text{or}\quad V_i (x) = \log \left( 1 + \frac{(x+1.5)^2}{2} \right) +C_{H},
    \end{align}
    where $C_{H}$ is a constant depending on the operator.
    For $\partial E_i = \emptyset$, we consider
    \begin{equation}
        V_i(x) = \log \left(1+\frac{1}{2}x^2\right).
    \end{equation}
    Note that all examples have semi-concavity constant $\kappa_{V_i} = 1$.
\end{example}

In the course of the proof, we use that the containment function $\bfV$ acts like a Lyapunov function for the Hamiltonian $H$, which means that the action of the operator on the containment function should be bounded. In the case that the underlying space is (partially) bounded in $E_i$, for some $i\in \{1,2\}$, we additionally require that the action of the operator is negative in that direction.

\begin{remark}\label{remark:Lyapunov_tradeoff}
    Note that, in the case that $E_i$ is unbounded, there is a tradeoff when choosing a containment function between allowing for stronger growth of the solution and still being able to control the action of the operator on the containment function, i.e.\ the requirement that the containment functions need to be a Lyapunov function. Consider, for example, the operator $Hf (x) = \ip{x}{\nabla f (x)}$, which has bounded action for $\log(1+x^2)$ but not $x^2$.
\end{remark}

The next definition contains the perturbations for the Jensen-type perturbation procedure, cf.\ Proposition \ref{proposition:Jensen_Alexandrov_cutoff}, which we use to find new optimizers, in which the second derivatives exist.\ 
As before, we use families of locally linear functions $\bfzeta_{p, z}$ as a first-order penalization and families of localization functions $\bfxi_z$ as second-order penalization.
The procedure is performed for $E_1$ and $E_2$ separately, the perturbations have the structure of a direct sum on $E$.
The structure of the penalizations is motivated by the usual prototypical examples of lines and parabolas, i.e.
\begin{align}
    \bfzeta_{z,p}(x) = \zeta_{1,z_1,p_1}(x_1) + \zeta_{2,z_2,p_2}(x_2) &= \ip{p_1}{x_1-z_1} + \ip{p_2}{x_2-z_2}, \label{eq:perturbation_1}\\
    \bfxi_z(x) = \xi_{1,z_1}(x_1) + \xi_{2,z_2}(x_2) &= \frac{1}{2} \big( d_1^2 (x_1,z_1) + d_2^2 (x_2, z_2) \big), \label{eq:perturbation_2}
\end{align}
centered at some $z \in E$ and with slopes $p=(p_1, p_2) \in \bR^{q_1 + q_2}$.

\begin{definition}[Point penalizations] \label{definition:perturbation_first_second_order}
    Let $i \in \{1,2\}$. We call collections of maps $\{\zeta_{i,z_i,p_i}\}_{z_i \in E_i, p_i \in \bR^{q_i}} \subseteq C(E_i)$ and $\{\xi_{i, z_i}\}_{z_i \in E_i} \subseteq C^1(E_i)$, $\zeta_{i,z_i,p_i} : E_i \rightarrow \bR$ and $\xi_{i,z_i} : E_i \rightarrow \bR$ sets of \emph{first} and \emph{second order point penalizations on $E_i$}, respectively, if there exist constants $R_i>0$ and $\kappa_{\xi_i} > 0$ such that for all $z_i \in E_i$:
    \begin{enumerate}[(a)]
        \item \label{item:definition:penalization:linear} $\zeta_{i,z_i,p_i}$ is linear in terms of $p_i$ around $z_i$:
            \begin{equation*}
                \zeta_{i,z_i,p_i}(y_i) = \ip{p_i}{y_i-z_i}
            \end{equation*}
            if $y_i \in B_{R_i}(z_i) \cap E_i$.
        \item \label{item:definition:penalization:semi-concave} The map $\xi_{i,z_i}$ is semi-concave with constant $\kappa_{\xi_i}$ and can be extended to a neighborhood $\cN(E_i)$ such that it is semi-concave with the same constant.
        \item \label{item:definition:penalization:second_order} The map $\xi_{i,z_i}$ is a penalization away from $z_i$:
            \begin{equation*}
                \xi_{z_i}(z_i) = 0, \qquad \xi_{z_i}(y_i) > 0, \qquad \text{if } y_i \neq z_i.
            \end{equation*}
        \item \label{item:definition:penalization:domination}  We have
            \begin{equation*}
                \inf_{|p_i| \leq 1} \inf_{y_i \in E_i \setminus B_{R_i}(z_i)} \xi_{i,z_i}(y_i) + \zeta_{i,z_i,p_i}(y_i) > 0.
            \end{equation*}
    \end{enumerate}

    We simply call $\{\bfzeta_{z,p}\}_{z \in E, p \in \bR^{q_1 + q_2}} \subseteq C(E)$ and $\{\bfxi_{z}\}_{z \in E} \subseteq C^1(E)$ with
    \begin{align}
        \bfzeta_{z,p}(x) &= \zeta_{1,z_1,p_1}(x_1) + \zeta_{2,z_2,p_2} (x_2),\\
        \bfxi_{z}(x) &= \xi_{1, z_1}(x_1) + \xi_{2, z_2} (x_2),
    \end{align}
    sets of \emph{first} and \emph{second order point penalizations}, if the collections $\{\zeta_{1,z_1,p_1}\}_{z_1 \in E_1, p_1 \in \bR^{q_1}}$, $\{\zeta_{2,z_2,p_2}\}_{z_2 \in E_2, p_2 \in \bR^{q_2}}$, $\{\xi_{1, z_1}\}_{z_1 \in E_1}$, and $\{\xi_{2, z_2}\}_{z_2 \in E_2}$ are sets of first and second order point penalizations on $E_1$ and $E_2$, respectively.
    
    For any given $z_0,z_1 \in E$ and $p \in \bR^q$, we consider the maps
    \begin{align}\label{definition:Jensen_penalization_Xi}
        \Xi^0(y) = \Xi_{z_0,p}^0(y) &\coloneqq \bfxi_{z_0}(y) + \bfzeta_{z_0,p}(y), \\
         \Xi(y) = \Xi_{z_0,p,z_1}(y) &\coloneqq \bfxi_{z_0}(y) + \bfzeta_{z_0,p}(y) + \bfxi_{z_1}(y).
    \end{align}
\end{definition}

For the examples in Section \ref{sec:applications}, we work with the following two choices for $\bfzeta$ and $\bfxi$. The first collection follows \eqref{eq:perturbation_1} and \eqref{eq:perturbation_2} and is a usual choice in the literature. The second collection is based on a cut-off of the first collection. We do this to control the action of non-local integral operators on the perturbations.

\begin{example}[Jensen penalization functions] \label{example:keyPenalizations}
\ Consider the following two collections of penalization functions:
    \begin{description}
    \item[Collection 1]\label{definition:canonical_penalization} The base penalizations are
          \begin{align*}
        \bfzeta_{z,p}(x) & = \ip{p_1}{x_1-z_1} + \ip{p_2}{x_2-z_2},\\
        \bfxi_{z}(x) & = \frac{1}{2} \bfd^2(x,z).
    \end{align*}
    \item[Collection 2] \label{definition:levy_penalizations}
    Let $R'' > R' > R >2 $. Let $\overline{\ell}: [0, \infty) \rightarrow [0, \infty)$ be a smooth function satisfying $\overline{l}(r) = 1$ for $r < R'$ and $\overline{l}(r) = 0$ for $x > R''$. 
    Let
    \begin{align*}
        \overline{\bfzeta}_{p, z} (x) &= \overline{\ell} (\bfd(x, z)) \big(\ip{p_1}{x_1-z_1} + \ip{p_2}{x_2-z_2}\big),\\
        \overline{\bfxi}_z (x) &= (1-\overline{\ell}(\bfd(x, z)))(R'' + 1)^2 + \overline{\ell}(\bfd(x, z)) \frac{1}{2} \bfd^2(x,z).
    \end{align*}
    \end{description}
\end{example}

\subsection{Comparison principle}\label{sec:main_thm}

The following theorem is a variant of \cite[Theorem 3.1]{DCFuKrNe24} for (partially) bounded spaces and the main result of this work. Our goal is to show comparison for a boundary value problem of form
\begin{equation}\label{eq:main_th_bvp}
    \begin{cases}
        \bH f(x)=0, & \text{if } x \in E,\\
        \bG f(x)=0, & \text{if } x \in \partial E,
    \end{cases}
    \tag{BVP}
\end{equation}
with $\bH \subseteq C(E) \times C(E)$ and some operator $\bG$. In many cases, like the parabolic problems considered in Section \ref{sec:parabolic_example}, we have $\bG \subseteq C(E) \times C(\partial E)$. However, if the boundary of $E$ itself has a subset that behaves like a boundary, denoted by $\partial\partial E$, $\bG$ can itself be a boundary value problem. In this case, we essentially treat a problem of form 
\begin{equation}\label{eq:main_th_bvp2}
    \begin{cases}
        \bH f(x)=0, & \text{if } x \in E,\\
        \bG f(x)=0, & \text{if } x \in \partial E,\\
        \bF f(x)=0, & \text{if } x \in \partial\partial E,
    \end{cases}
    \tag{BVP'}
\end{equation}
where now $\bH \subseteq C(E) \times C(E)$, $\bG \subseteq C(E) \times C(\partial E)$, and $\bF \subseteq C(E) \times C(\partial \partial E)$, see Section \ref{sec:elliptic_quardant}.

In the viscosity sense, the treatment of both \eqref{eq:main_th_bvp} and \eqref{eq:main_th_bvp2} break down to (iteratively) showing comparison between subsolutions $u$ to the subsolution operator
\begin{equation}\label{eq:main_th_subsol_op}
    \bH_1 f(x) = \begin{cases}
        \bH f(x), & \text{if } x \in E,\\
        -\min\{-\bG_1 f(x), -\bH f(x)\}, & \text{if } x \in \partial E,
    \end{cases}
\end{equation}
where $\bH_1 \subseteq C(E) \times \big(C(\interior{E}) \cap \mathrm{LSC}(E)\big)$,
and supersolutions $v$ to the supersolution operator
\begin{equation}\label{eq:main_th_supersol_op}
    \bH_2 f(x) = \begin{cases}
        \bH f(x), & \text{if } x \in E,\\
        -\max\{-\bG_2 f(x), -\bH f(x)\}, & \text{if } x \in \partial E,
    \end{cases}
\end{equation}
where $\bH_2 \subseteq C(E) \times \big(C(\interior{E}) \cap \mathrm{USC}(E)\big)$
with $\bG_1 \subseteq C(\partial E) \times \mathrm{LSC}(\partial E)$, $\bG_2 \subseteq C(\partial E) \times  \mathrm{USC}(\partial E)$, and $\bH \subseteq C(E) \times C(E)$ of the form
\begin{align*}
    \bH f(x) & = \sup_{\theta \in \Theta} \big\{\bA_{\theta} f(x) + \bB_{\theta} f(x) - \cI(x,\theta) \big\}.
\end{align*}

Theorem \ref{th:comparison_HJI} below provides a comparison principle for such an inductive step.

\begin{remark}
    We retain the minus signs in front of the minimum and maximum to remain consistent with the notation used in \cite{DCFuKrNe24}.
\end{remark}

\begin{theorem}[Strict comparison principle]\label{th:comparison_HJI}
    Consider a doubling penalization $\Phi$, a containment function $\bfV$, and penalization functions $\{\bfzeta_{z,p}\}_{z \in E, p \in \bR^{q_1+q_2}}$ and $\{\bfxi_z\}_{z \in E}$. Let $\bH_1 \subseteq C(E) \times \big(C(\interior{E}) \cap \mathrm{LSC}(E)\big)$ as in \eqref{eq:main_th_subsol_op} and $\bH_2 \subseteq C(E) \times \big(C(\interior{E}) \cap \mathrm{USC}(E)\big)$ as in \eqref{eq:main_th_supersol_op} with $\bH \subseteq C(E) \times C(E)$ in Hamilton--Jacobi--Bellman form
    \begin{align}
        \bH f(x) &= \sup_{\theta \in \Theta} \big\{\bA_{\theta} f(x) + \bB_{\theta} f(x) - \cI(x,\theta) \big\}
    \end{align}
    with $\Theta$ a compact, metric space, and boundary operators $\bG_1 \subseteq C(\partial E) \times \mathrm{LSC}(\partial E)$ and $\bG_2 \subseteq C(\partial E) \times \mathrm{USC}(\partial E)$ all satisfying the technical Assumptions \ref{assumption:domain_setup} and \ref{assumption:compatibility} below, and 
    \begin{enumerate}[(a)]
        \item \label{item:assumption:Acoupling} For all $\theta \in \Theta$, $\bA_{\theta}$ is linear on its domain and has a $\Phi$-controlled growth coupling $\widehat{\bA}_{\theta}$ as in Definition \ref{def:coupling} with a modulus uniform in $\theta$. 

        \item \label{item:assumption:Bmonotone} For all $\theta \in \Theta$, $\bB_{\theta}$ is convex semi-monotone operators w.r.t.\ $\Phi$ as in Definition \ref{definition:first_order} with a modulus uniform in $\theta$.      
        \item \label{item:assumption_directHJI_lsc} The cost functional $\cI \colon E \times \Theta \rightarrow (-\infty, \infty]$ is lower semi-continuous in $(x,\theta)$ and admits a modulus of continuity in $x$ uniformly in $\theta$.

        \item \label{item:assumption_Lyapunovfunction} For $i\in \{1,2\}$, $V_i$ is a \emph{Lyapunov function} for $\bH$: We have $V_i \in \cD(\bH)$ and
        \begin{equation}\label{eq:mainth_HJI_bound}
            c_{V_i} \coloneqq \sup_{x \in E} \bH V_i(x) < \infty.
        \end{equation}
        If $\partial E_i \neq \emptyset$, we additionally have $c_{V_i} < 0$.
    \end{enumerate}

    \noindent Consider the pair of equations
    \begin{align} 
        -\bH_1 f \leq 0, \label{eqn:mainTheoremHJ_1}\\
        -\bH_2 f \geq 0. \label{eqn:mainTheoremHJ_2}
    \end{align}
    Let $u$ and $v$ be sub- and supersolutions to \eqref{eqn:mainTheoremHJ_1} and \eqref{eqn:mainTheoremHJ_2}, respectively, if $E_2$ is unbounded, with $u,v \in o(V_2)$.
    Then, for any compact set $K \subseteq E$ and $\eps = (\eps_1, \eps_2)$ with $\eps_1 + \eps_2 \in (0,1)$, we have
    \begin{equation}\label{eq:th:final}
        \sup_{x\in K} u(x) - v(x)  \leq  \varepsilon \widetilde{C}^K + \sup_{x \in \widetilde{K}} u(x) - v(x),
    \end{equation}
    where $\widetilde{K} = \widetilde{K}(K, \eps, u,v) \subseteq \partial E$ is a compact set given by
    \begin{multline*}
        \widetilde{K} \coloneqq \left\{z \in \partial E \, \middle| \,  \frac{1}{1\boxminus\eps} \left(\eps \bfV(z) - u(z)\right) + \frac{1}{1\boxplus\eps} \left(\eps \bfV(z) - v(z) \right)   \right. \\
        \leq \left. \frac{\eps}{1\boxminus\eps} \ssup{\bfV-u}_K + \frac{\eps}{1\boxplus\eps} \ssup{\bfV-v}_K - \ssup{u-v}_K \right\}
    \end{multline*}
    and $\widetilde{C}^K = (\widetilde{C}_{1}^K, \widetilde{C}_{2}^K)$ with 
    \begin{equation}
        \widetilde{C}_{1}^K=\widetilde{C}_{1}^K(u-V_1,v-V_1) \quad\text{and}\quad \widetilde{C}_{2}^K=\widetilde{C}_{2}^K(u-V_2,v-V_2).
    \end{equation}
    In particular, the strict comparison principle holds for \eqref{eqn:mainTheoremHJ_1} and \eqref{eqn:mainTheoremHJ_2}. 
\end{theorem}

The proof of Theorem \ref{th:comparison_HJI} is contained in Section \ref{section:main_proof} with the construction of the necessary optimizers and test functions in Section \ref{sec:opt_construction}. Applications of the above result to parabolic equations and elliptic equations on spaces with corners can be found in Section \ref{sec:applications}.

\subsection{Regularity and compatibility assumptions} \label{subsection:regularity_and_compatibility}

In this section, we state the technical assumptions necessary for the proof the main theorem.

As we have a choice for the domain of our operator and only need functions with compact sub- and superlevel sets, we need to ensure that the domains of the restrictions are regular enough to perform our analysis. In particular, the action of the operator on test functions and their combinations with perturbations needs to be well-defined. 

\begin{assumption}[Regularity of $\bH_1$ and $\bH_2$] \label{assumption:domain_setup}
    Let $\bH_1 \subseteq C(E) \times \big(C(\interior{E}) \cap \mathrm{LSC}(E)\big)$ be as in equation \eqref{eq:main_th_subsol_op} and $\bH_2 \subseteq C(E) \times \big(C(\interior{E}) \cap \mathrm{USC}(E)\big)$ as in equation \eqref{eq:main_th_supersol_op} be operators with the following two restrictions
    \begin{align*}
        H_+ & \subseteq \left\{(f,g) \in \bH_1 \, \middle| \, f \in C_+(E) \right\}, \\
        H_- & \subseteq \left\{(f,g) \in \bH_2 \, \middle| \, f \in C_-(E) \right\},
    \end{align*}
    satisfying
    \begin{enumerate}[(a)]
        \item \label{item:assumption_domain_setup:max_principle} $\bH$ satisfies the maximum principle,
        \item \label{item:assumption_domain_setup:domainH} for $j \in \{1,2\}$, $\cD(\bH_j)$ is a cone and $C_c^\infty(E) \subseteq \cD(\bH_j) \subseteq C(E)$,
        \item \label{item:domain_setup_extended_convex} $\cD(H_+)$ is convex,
        \item \label{item:domain_setup_extended_affine} 
        for any $f_1 \in \cD(\bH_1)$, $f_2 \in \cD(\bH_2)$, and $g \in \cD(H_+)$ and $\delta \in (0,1)$ we have
        \begin{align*}
            & (1-\delta) f_1 + \delta g \in \cD(H_+),
            & (1+\delta) f_2 - \delta g \in \cD(H_-).
        \end{align*}
    \end{enumerate}
\end{assumption}

In the main theorem, Theorem \ref{th:comparison_HJI}, we assume that the interior operator $\bH$ is of HJB type.\ To perform our analysis, we need to require that that the action of the constituting operators $\bA_\theta$ and $\bB_\theta$ as well as the boundary operators $\bG_1$ and $\bG_2$ on the functions, we use to perform our analysis, is sufficiently regular in the compatibility sense of the following assumption.

\begin{assumption}[Compatibility of $\bA_{\theta}$, $\bB_{\theta}$, $\bG_1$, and $\bG_2$] \label{assumption:compatibility}
Let $\Theta$ be a compact, metric space. For $\theta\in \Theta$, let $\bA_{\theta}, \bB_{\theta} \subseteq C(E) \times C(E)$.
Consider a containment function $\bfV$ as in Definition \ref{definition:perturbation_containment} and penalization functions $\{\bfzeta_{z,p}\}_{z \in E, p \in \bR^{q_1+q_2}}$ and $\{\bfxi_z\}_{z \in E}$ as in Definition \ref{definition:perturbation_first_second_order}.
  \begin{enumerate}[(a)]
        \item\label{item:compatA} Let the collection $\{\bA_{\theta}\}_{\theta \in \Theta}$ be \emph{compatible} with $\bfV$, $\{\bfzeta_{z,p}\}_{z \in E, p \in \bR^{q_1+q_2}}$, and $\{\bfxi_z\}_{z \in E}$, i.e.,
        \begin{enumerate}[(1)]
            \item \label{item:compatA:domain} we have
            \begin{align}
                \bfV \circ s_z \in \cD(\bA_{\theta}),\quad
                \Xi_{z_0,p,z_1} \circ s_z \in \cD(\bA_{\theta}),
            \end{align}
            for any $\theta \in \Theta$ and $z \in \overline{B_1(0)} \cap \{E-x\}$,
            \item \label{item:compatA:cont} the maps
            \begin{align*}
            (\theta,x,z_0,p,z_1,z) & \mapsto \bA_{\theta} \left( \Xi_{z_0,p,z_1} \circ s_z\right)(x),\\
            (\theta,x,z) & \mapsto \bA_{\theta}\left(\bfV \circ s_z\right)(x)
            \end{align*}
            are continuous, 
            \item \label{item:compatA:contf} the map
            \begin{equation*}
                \theta \mapsto \bA_{\theta} f(x)
            \end{equation*}
            is continuous for any $x \in E$ and $f \in \bigcap_{\theta \in \Theta} \cD(\bA_{\theta})$. 
        \end{enumerate}

        \item\label{item:compatB} Let the collection $\{\bB_{\theta}\}_{\theta \in \Theta}$ be \emph{compatible} with $\bfV$, $\{\bfzeta_{z,p}\}_{z \in E, p \in \bR^{q_1+q_2}}$, and $\{\bfxi_z\}_{z \in E}$, i.e.,
        \begin{enumerate}[(1)]
            \item \label{item:compatB:domain} 
            we have
            \begin{align}
                \bfV \circ s_z \in \cD(\bB_{\theta}),\quad
                \Xi_{z_0,p,z_1} \circ s_z \in \cD(\bB_{\theta}),
            \end{align}
            for any $\theta \in \Theta$ and $z \in \overline{B_1(0)}\cap \{E-x\}$,
            \item \label{item:compatB:cont} the maps
            \begin{align*}
            (\theta,x,z_0,p,z_1) & \mapsto \bB_{\theta} \Xi_{z_0,p,z_1} (x),\\
            (\theta,x) & \mapsto \bB_{\theta} \bfV(x)
            \end{align*}        
            are continuous, 
            \item \label{item:compatB:contf} the map
            \begin{equation*}
                \theta \mapsto \bB_{\theta} f(x)
            \end{equation*}
            is continuous for any $x \in E$ and $f \in \bigcap_{\theta \in \Theta} \cD(\bB_{\theta})$. 
        \end{enumerate}

        \item\label{item:compatG} Let $\bG_1$ and $\bG_2$ be \emph{compatible} with $\bfV$, $\{\bfzeta_{z,p}\}_{z \in E, p \in \bR^{q_1+q_2}}$, and $\{\bfxi_z\}_{z \in E}$, i.e., for $j\in \{1,2\}$, let
        \begin{equation}
            \bfV \circ s_{z} \in \cD(\bG_j) \quad\text{and}\quad \Xi_{z_0,p,z_1} \circ s_{z} \in \cD(\bG_j).
        \end{equation}
    \end{enumerate}
\end{assumption}

\section{Applications}\label{sec:applications}
In this section, we apply the general framework to two examples. In Section \ref{sec:parabolic_example}, we consider a parabolic partial differential equation. In Section \ref{sec:elliptic_quardant}, we consider an elliptic equation on a space with a corner.

\subsection{Parabolic equations}\label{sec:parabolic_example}
 In this section, we treat stochastic parabolic problems on $E = [0, T] \times \bR^q$ with $q\in \bN$. To avoid confusion, in this section we denote an element from $E$ as $(t,x)$. Thus, we consider operators of type
\begin{align*}
    \frac{\partial f}{\partial t} (t,x) &= \cH [f(t, \cdot)] (x) ,\\
    f(T,x) &= f_T(x),
\end{align*}
where, for some compact $\Theta$, the operator $\cH \subseteq C(E) \times C(E)$ is such that the following is well-defined
\begin{align}
    \cH [f(t, \cdot)](x) &= \sup_{\theta \in \Theta} \Big\{ \widetilde{\bB}_\theta [f(t,\cdot)](x)+ \widetilde{\bA}_\theta [f(t,\cdot)](x)  - \cI(\theta)\Big\}\\
    &= \sup_{\theta \in \Theta} \bigg\{\!\ip{b(x, \theta)}{\nabla f(t,x)} + \frac{1}{2} \Tr \left( \Sigma\Sigma^T(x,\theta) D^2 f(t, x) \right)\\
    &\hspace{2cm}+ \int \left[f(t,x+\bfz) - f(t,x) - \rchi_{B_1(0)} (\bfz) \ip{\bfz}{\nabla f(t,x)}\right] \mu_{x,\theta}(\dd \bfz)\bigg\},
\end{align}
and the coefficients $b$ and $\Sigma$ satisfying the usual Lipschitz and linear growth conditions, see, e.g., \cite[Section IV]{FlSo06},
and $f_T \in C_b(\bR^q)$.
For exposition, we require the jump measures $x \mapsto \mu_{x, \theta}$ to be L\'evy measures with finite second moment and admitting a Lipschitz extended coupling $\pi_{x,x'}$ in the sense that, for any compact set $K \subseteq \domain$, there exits a constant $L_{\pi,K}$ such that, for $x,x' \in K$, we have
\begin{equation*}
    \int  d^2(\bfz_1,\bfz_2) \pi_{x,x'}( \dd \bfz_1, \dd \bfz_2) \leq L_{\pi,K} d^2(x,x').
\end{equation*}
For a more detailed discussion of such families of jump measures, we refer to \cite[Section 4.3]{DCFuKrNe24}.

Consequently, for $i\in\{1,2\}$, we have 
\begin{equation}\label{ex:eq:parabolic_operator}
\begin{aligned}
    \bH f(t,x)&  = \frac{\partial f}{\partial t} (t, x) - \cH[f(t, \cdot)](x), &&\text{if } (t,x) \in [0, T] \times \bR^q,\\
    \bG_i f(t,x) & = \frac{\partial f}{\partial t} (t, x) - \cH[f(t, \cdot)](x), &&\text{if } (t,x) \in \{0\} \times \bR^q,\\
    \bG_i f(t,x) & = f_{T,i}(x) - f(T,x) , &&\text{if } (t,x) \in \{T\} \times \bR^q.
\end{aligned}
\end{equation}
with $f_{T,i} \in C_b(\bR^q)$. The subsolution operator $\bH_1$ is then defined as
\begin{equation}
    \bH_1 f(x) = 
    \begin{cases}
        \tfrac{\partial f}{\partial t} (t, x) - \cH[f(t, \cdot)](x), &\text{if } (t,x) \in [0, T) \times \bR^q,\\
        -\!\min\big\{ \!-\!\tfrac{\partial f}{\partial t} (t, x) + \cH[f(t, \cdot)](x), f(T,x)- f_{T,i}(x)\big\} &\text{if } (t,x) \in \{T\} \times \bR^q,
    \end{cases}
\end{equation}
with the supersolution operator analogously.

Parabolic equations of the type \eqref{ex:eq:parabolic_operator} model a wide class of finite time horizon Cauchy problems. In Economics and Finance, they are commonly used to describe, e.g., utility maximization and other optimal control problems, see \cite{FlSo06,Ph09,YoZh99} for an overview in a finite-dimensional and \cite{FaGoSw17} in an infinite-dimensional context.
Another application of comparison results for this type of equation is the framework of $G$-L\'evy processes, cf.\ \cite{hu.2021,Pe07,neufeld.2017}, which is used to model, i.a., derivative pricing under model uncertainty.

To show how this setting fits into our framework, we need to specify the doubling penalization $\Phi$ as well as appropriate Lyapunov functions $V_1$ and $V_2$ and penalizations $\bfzeta_{z, p}$ and $\bfxi_z$.

As is standard in the literature, we choose 
\begin{equation}
    \alpha\Phi((t,y), (t',y')) = \frac{\alpha_1}{2}(t-t')^2 + \frac{\alpha_2}{2}d^2(y, y')
\end{equation}
as the doubling penalization and
\begin{equation}\label{eq:ex:parabolic_Lyapunov}
    \bfV(t,x) = V_1 (t) + V_2(x) = (T- t) + \log\big( 1 + \tfrac{1}{2} |x|^2 \big)
\end{equation}
as the containment function. 
In fact, the choices described above are the prototypical examples for doubling penalizations and Lyapunov functions on a cylindrical space of this type.

Note, that in the context of the comparison proof the time derivative can be treated as a simple drift term such that the operator $\bH$ consists of
\begin{equation}
    \bB_\theta f(t,x) = -\frac{\partial f}{\partial t} (t,x) + \widetilde{\bB}_\theta [f(t,\cdot)](x),\quad
    \bA_\theta f(t,x) =  \widetilde{\bA}_\theta [f(t,\cdot)](x),
\end{equation}

Using this insight, we find that the appropriate collection of Jensen penalizations, as in Example \ref{example:keyPenalizations} but applied to the parabolic problem, are
\begin{align*}
    \zeta_{z,p}(t, x) & = \ip{p_1}{t-z_1} + \ip{p_2}{x - z_2},\\
    \xi_{z}(t, x) & = \frac{1}{2} \Big((t-z_1)^2 + d^2(x,z_2)\Big).
\end{align*}
Note that due to the assumed existence of second moments of the jump measures, we do not need to consider the cut-off versions of the penalizations.

It thus remains to verify that the containment function $\bfV$ in equation \eqref{eq:ex:parabolic_Lyapunov} is a Lyapunov function for $\bH$. However, for any $t \in [0,T]$, we have $\bH V_1 (t) = -1<0$ and direct calculations, cf.\ \cite[Section 4]{DCFuKrNe24}, show that $\sup_{x \in \bR^q} \bH V_2 (x) < \infty$, as required.

Now, we can apply Theorem \ref{th:comparison_HJI} to find that, for any compact set $K \subseteq E$, there exist a compact set $\widetilde{K}_x$ and constants $\widetilde{C}^K = (\widetilde{C}_{1}^K, \widetilde{C}_{2}^K)$ such that we have
\begin{align}
    \sup_{(t,x) \in K} u(t,x) - v(t,x) &\leq \eps \widetilde{C}^K + \sup_{x \in \widetilde{K}_x} u(T,x) - v(T,x)
\end{align}
Now taking $\eps_1, \eps_2 \downarrow 0$ and inserting the terminal conditions, we find that
\begin{equation}
    \sup_{(t,x)\in [0,T]\times \bR^q} u(t,x) - v(t,x) \leq \sup_{x \in \bR^q} f_{T,1}(x) - f_{T,2}(x).
\end{equation}

\subsection{Spaces with corners}\label{sec:elliptic_quardant}

In this section, we treat elliptic problems on the first quadrant $E = [0, \infty) \times [0, \infty)$, a space that, in particular, has a corner. For $\lambda>0$, $h \in C_b(E)$, $h^{\partial_1} \in C_b(\partial_1 E)$, and $h^{\partial_2} \in C_b(\partial_2 E)$, we consider the boundary value problem
\begin{align}
    f(x) - \lambda \cH f(x) = h(x), &\quad\text{if } x\in E \coloneqq [0, \infty)\times[0, \infty),\\
    f(x) - \lambda \cH^{\partial_1} f(x)=h^{\partial_1}(x), &\quad\text{if } x\in \partial_1 E \coloneqq \{0\}\times[0, \infty),\\
    f(x) - \lambda \cH^{\partial_2} f(x)=h^{\partial_2}(x), &\quad\text{if } x\in \partial_2 E \coloneqq [0, \infty)\times\{0\},\\
    f(x) = 0, &\quad\text{if } x\in \partial_{0} E \coloneqq \{0\}\times\{0\},
\end{align}
with $\cH \subseteq C(E) \times C(E)$ such that the following is well-defined
\begin{multline}
    \cH f(x) = \ip{b(x)}{\nabla f(x)} + \frac{1}{2} \Tr \left( \Sigma\Sigma^T(x) D^2 f(x) \right)\\ + \int \left[f(x+\bfz) - f(x) - \rchi_{B_1(0)} (\bfz) \ip{\bfz}{\nabla f(x)}\right] \mu_{x}(\dd \bfz),
\end{multline}
and, for exposition, $b$ and $\Sigma$ Lipschitz and bounded and $\mu_x \in \cM_W([-x_1, \infty)\times [-x_2, \infty))$ and Lipschitz as in \cite[Section 4.3]{DCFuKrNe24}
and, for $i \in \{1,2\}$, the boundary operators $\cH^{\partial_i} \subseteq C(E) \times C(\partial_i E)$ of form
\begin{equation}
    \cH^{\partial_i} f(x) = \ip{b^{\partial_i}(x)}{\nabla f(x)} + \int \left[f(x+\bfz) - f(x) - \rchi_{B_1(0)} (\bfz) \ip{\bfz}{\nabla f(x)}\right] \mu_{x}^{\partial_i}(\dd \bfz)
\end{equation}
with $b^{\partial_i}$ Lipschitz, bounded, and oblique to $E$, i.e., $\ip{b^{\partial_i} (x)}{n(x)}> 0$ for all $x \in \partial_i E$ and $n(x)\in N_x$. Furthermore, we consider $\mu_{x}^{\partial_i}\in \cM_W(\partial_i E \cap ([-x_i, \infty)\times \{0\}))$, Lipschitz, and jumping inwards, i.e., $\text{supp}(\mu_{x}^{\partial_i})\subseteq [-x_i, 0]\times \{0\}$.

Consequently, the subsolution operator $\bH_1$ is of form
\begin{equation}
    -\bH_1 f(x) = 
    \begin{cases}
    f(x) - \lambda \cH f(x)-h(x), &\text{if } x\in \interior{E},\\
    \min\{f(x) - \lambda \cH^{\partial_1} f(x)-h^{\partial_1}(x), f(x) - \lambda \cH f(x)-h(x)\}, &\text{if } x\in \partial_1 E,\\
    \min\{f(x) - \lambda \cH^{\partial_2} f(x)-h^{\partial_2}(x),  f(x) - \lambda \cH f(x)-h(x)\}, &\text{if } x\in \partial_2 E,\\
    \min\{f(x), f(x) - \lambda \cH^{\partial_1} f(x)-h^{\partial_1}(x),\\
    \qquad f(x) - \lambda \cH^{\partial_2} f(x)-h^{\partial_2}(x), f(x)- \lambda \cH f(x)-h(x)\},  &\text{if } x\in \partial_0 E.
    \end{cases}
\end{equation}
The supersolution operator $-\bH_2$ is then defined analogously. 
As the doubling penalization, we choose $\alpha\Phi(x, x') = \frac{\alpha}{2}d^2(x_1, x'_1) + \frac{\alpha}{2}d^2(x_2, x'_2)$, where $d$ is the Euclidean distance on $\bR$. As Jensen penalizations, we choose the canonical examples as in Example \ref{example:keyPenalizations}.
An appropriate Lyapunov function in the above case is the prototypical example $\log(1+\frac{x^2}{2})$ with a slight modification that ensures that $\sup_{x\in E}\bH_j \bfV(x) < 0$ holds for $j \in \{1,2\}$: We choose $\bfV = V_1 + V_2$ with
\[
    V_i (x_i) = \log \left( 1 + \frac{(x_i+1.5)^2}{2} \right)+ \max\{\|h\|, \|h^{\partial_1}\|, \|h^{\partial_2}\|\}+1.
\]

Analogous to \cite[Section 4]{DCFuKrNe24}, direct calculations show that $\sup_{x \in E}\cH \bfV (x) < \infty$, $\sup_{x \in \partial_1 E}\cH^{\partial_1} \bfV (x) < \infty$, and $\sup_{x \in \partial_2 E}\cH^{\partial_2} \bfV (x) < \infty$.
Additionally, we find that
\begin{gather}
    \sup_{x \in E} \ip{b(x)}{\nabla \bfV(x)} \leq c_{\bfV, b} < 0,\\
    \sup_{x \in E} \Tr \left( \Sigma\Sigma^T(x) D^2 \bfV(x) \right) \leq c_{\bfV, \Sigma} < 0,\\
    \sup_{x \in E} \int \left[\bfV(x+\bfz) - \bfV(x) - \rchi_{B_1(0)} (\bfz) \ip{\bfz}{\nabla \bfV(x)}\right] \mu_{x}(\dd \bfz) \leq c_{\bfV, \mu} < 0,
\end{gather}
and the same for their boundary counterparts.\ 
Consequently, we find that $\bfV$ is a Lyapunov function for $\bH$.

Now, using the exterior sphere condition and the fact that $\bfV$ is a Lyapunov function, well-known results, cf.\ \cite{CIL92,DuIs91,IsKu22,DuIs90,Is91}, yield that the first-order parts of $\cH$, $\cH^{\partial_1}$, and $\cH^{\partial_2}$ are convex semi-monotone.

The coupling of the stochastic part of the interior operator
\begin{equation}
    \bA f(x) \coloneqq \frac{1}{2} \Tr \left( \Sigma\Sigma^T(x) D^2 f(x) \right) + \int \left[f(x+\bfz) - f(x) - \rchi_{B_1(0)} (\bfz) \ip{\bfz}{\nabla f(x)}\right] \mu_{x}(\dd \bfz)
\end{equation}
then works as it did in \cite[Sections 4.2 and 4.3]{DCFuKrNe24}: To couple $\frac{1}{2} \Tr \left( \Sigma\Sigma^T(x) D^2 f(x) \right)$ we consider the operator
\begin{equation*}
    \widehat{\bA}_\Sigma g(x,x') \coloneqq \tfrac12 \Tr \big(\widehat{\Sigma}^2(x,x') D^2 g(x,x') \big),
\end{equation*}
with
\begin{equation*}
    \widehat{\Sigma}^2(x,x') \coloneqq \begin{pmatrix}
        \Sigma(x)\Sigma^T(x) & \Sigma(x')\Sigma^T(x) \\
        \Sigma(x)\Sigma^T(x') & \Sigma(x')\Sigma^T(x') 
    \end{pmatrix},
\end{equation*}
which by \cite[Proposition 4.5]{DCFuKrNe24} and the Lipschitzianity and boundedness of $\Sigma$ is a $\bfd^2$-controlled growth coupling.

Analogously, using that $\mu_x \in \cM_W([-x_1, \infty)\times [-x_2, \infty))$ and Lipschitz, by \cite[Proposition 4.13]{DCFuKrNe24} we find a $\bfd^2$-controlled growth coupling of
\begin{equation}
    \int \left[f(x+\bfz) - f(x) - \rchi_{B_1(0)} (\bfz) \ip{\bfz}{\nabla f(x)}\right] \mu_{x}(\dd \bfz)
\end{equation}
that is of form
\begin{multline}
    \widehat{\bA}_\mu g(x,x') \coloneqq\int \Big[ g(x+\bfz_1,x'+\bfz_2) - g(x,x')\\
    - \widehat{\rchi}(\bfz_1,\bfz_2) \ip{(\bfz_1, \bfz_2)^T}{\nabla g(x, x')}\Big] \pi_{x,x'}(\dd \bfz_1, \dd \bfz_2).
\end{multline}
Consequently, we find that
\begin{equation}
    \widehat{\bA} g(x,x') = \widehat{\bA}_\Sigma g(x,x') + \widehat{\bA}_\mu g(x,x')
\end{equation}
is a $\bfd^2$-controlled growth coupling of $\bA$, which overall allows us to apply Theorem \ref{th:comparison_HJI} to show the strict comparison principle for $\bH_1$ and $\bH_2$.

Note that we can now reapply the strict comparison principle to $\partial_1 E$ and $\partial_2 E$ with $\partial_0 E = \{0\} \times \{0\}$ as the boundary.\ Restricted versions of $\bfV$ and the collections of Jensen penalizations, cf.\ Example \ref{example:keyPenalizations}, are still appropriate as all required properties hold on the entirety of $E$. The arguments for convex semi-monotonicity and the coupling are as outlined above.
Thus, for any compact set $K \subseteq E$, resulting compact set $\widetilde{K} \subseteq \partial_0 E \cup \partial_1 E \cup \partial_2 E$, and constants $\widetilde{C}^K, \widetilde{C}^{\widetilde{K}} \in \bR^2$, we get the following chain of inequalities:
\begin{equation}
    \sup_{x\in K} u(x) - v(x)  \leq  \varepsilon \widetilde{C}^K + \sup_{x \in \widetilde{K}} u(x) - v(x) \leq \varepsilon \widetilde{C}^{\widetilde{K}} + \big(u(0) - v(0)\big) = \eps \widetilde{C}^{\widetilde{K}}.
\end{equation}
Again taking $\eps_1, \eps_2 \downarrow 0$, we find that
\begin{equation}
    \sup_{x\in E} u(x) - v(x) \leq 0.
\end{equation}

\begin{remark}\label{remark:Lyapunov_coeffs}
    Note that the condition that an appropriate Lyapunov function $\bfV$ needs to exist, cf.\ Theorem \ref{th:comparison_HJI} \ref{item:assumption_Lyapunovfunction}, essentially enforces what types of operators on spaces with boundaries we can treat.\ An example of this in the above context is that, if we want to treat a diffusion term on a quadrant, the drift terms on the boundary need to be oblique to $E$, i.e., we require a condition of type $\ip{b^{\partial_i}(x)}{n(x)}>0$ to hold, and the jump term on the boundary cannot jump out too far, i.e., we require that $\text{supp}(\mu_{x}^{\partial_i})\subseteq [-x,0]$. Otherwise the existence of a Lyapunov function fails. 
    If the operator does not contain any diffusion terms, the conditions on the drift and jump terms can be relaxed.
\end{remark}
\section{Construction of optimizers}\label{sec:opt_construction}
As in other comparison proofs, we perform variable quadruplication for the optimization problem $\sup (u-v)$.
Compared to the proof in \cite{DCFuKrNe24}, we slightly adjust our strategy for three reasons: 
Firstly, we now treat equations on the cylindrical space $E = E_1 \times E_2$. In general, the strategy for these types of spaces is to perturb and perform our analysis on $E_1$ and $E_2$ separately. 
In the case that one $E_i$ is unbounded, we additionally want to allow for the treatment of unbounded solution. The allowed growth in $E_i$ is in terms of growth of the Lyapunov function $V_i$. More specifically, since we penalize with a term of order $\eps V_i$, we need that, for every $\eps>0$, the sub- and supersolutions less $V_i$ are bounded from above, i.e., we require $u,v \in o(V_i)$.
Lastly, we replace the usual doubling-of-variables penalization $\bfd^2$ by a more general function $\Phi$. This replacement is motivated, for example, by the treatment of non-Lipschitz drifts, where one would like to penalize with, e.g., an Osgood function rather than a distance-squared-type object, see \cite[Page 587]{CrIsLi87}.

We begin, however, with the definitions of the sup- and inf-convolutions. For readability, we express suprema and infima using $\ssup{\cdot}$ and $\iinf{\cdot}$, respectively.

\begin{definition}[$\sup$- and $\inf$-convolution]\label{def:convolutions}
    Let $u: E \rightarrow \bR$ be upper semi-continuous and $v: E \rightarrow \bR$ be lower semi-continuous. For $\alpha = (\alpha_1, \alpha_2)$ with $\alpha_1, \alpha_2 > 1$, we define the \emph{$\sup$-convolution $P^\alpha[u]$ of $u$} as
    \begin{equation}\label{eq:def:sup-conv}
        P^\alpha[u](y) \coloneqq \sup_{x \in E} \left\{u(x) - \frac{\alpha}{2} \bfd^2(x,y) \right\} = \ssup{u - \frac{\alpha}{2} \bfd^2 (\cdot, y)}.
    \end{equation}
    Analogously, we define the \emph{$\inf$-convolution $P_\alpha[v]$ of $v$} as 
    \begin{equation}\label{eq:def:inf-conv}
        P_\alpha[v](y) \coloneqq \inf_{x \in E} \left\{v(x) + \frac{\alpha}{2} \bfd^2(x,y) \right\} = \iinf{u + \frac{\alpha}{2} \bfd^2 (\cdot, y)}.
    \end{equation}
\end{definition}

Now implementing the novelties described above, we perform a doubling-of-variables procedure: For $\alpha = (\alpha_1, \alpha_2)$ with $\alpha_1, \alpha_2 > 1$, we estimate
\begin{equation} \label{eqn:duplicationvariables_1}
    \sup_{x \in E} u(x) - v(x) \leq \sup_{x,x' \in E} u(x) - v(x') - \alpha \Phi(x,x')
\end{equation}
and then incorporate, for small $\eps = (\eps_1, \eps_2)$ with $\eps_1, \eps_2 > 0$, the containment functions $\bfV$ and upper bound $\sup (u-v)$ up to a term of order $\eps$. For the sake of clarity, we expand the entire equation once at this point and then only when necessary:
\begin{align} \label{eqn:duplicationvariables}
    &\sup_{x \in E} \frac{1}{1 \boxminus \varepsilon} u(x) - \frac{1}{1 \boxplus \varepsilon} v(x) \\
    &\quad\leq \sup_{x,x' \in E} \frac{1}{1\boxminus\varepsilon} u(x) - \frac{1}{1\boxplus\varepsilon} v(x') - \alpha \Phi(x,x') - \frac{\varepsilon}{1\boxminus\varepsilon}\bfV(x) - \frac{\varepsilon}{1\boxplus\varepsilon}\bfV(x')\\
    &\quad= \sup_{x,x' \in E} \frac{1}{1-\varepsilon_1-\varepsilon_2} u(x_1, x_2) - \frac{1}{1+\varepsilon_1+\varepsilon_2} v(x'_1, x'_2) - \alpha_1 \phi_1(x_1,x'_1)\\
    &\qquad -\alpha_2 \phi_2(x_2,x'_2) - \frac{\varepsilon_1}{1-\varepsilon_1-\varepsilon_2}V_1(x_1) - \frac{\varepsilon_2}{1-\varepsilon_1-\varepsilon_2}V_2(x_2)\\
    &\qquad- \frac{\varepsilon_1}{1+\varepsilon_1+\varepsilon_2}V_1(x'_1) - \frac{\varepsilon_2}{1+\varepsilon_1+\varepsilon_2}V_2(x'_2).
\end{align}

In Proposition \ref{prop:optimizer_construction} below, we perform an additional smoothing step that is necessary to treat second-order or integral operators, $\bA$ in our notation. This additional step extends the estimate from variable doubling to variable quadrupling.

Note that from this point on, w.l.o.g., we work with the assumption that $E_1$ is bounded. For $E_2$ we make case distinctions when necessary.

Recall that derivatives on boundaries are understood in the classical sense as a restriction of an extension of all functions to a neighborhood $\cN (E)$.

\begin{proposition}[Construction of optimizers]\label{prop:optimizer_construction}
    Let $\bfV$ be a containment function as in Definition \ref{definition:perturbation_containment}, $u$ be upper semi-continuous, $v$ be lower semi-continuous, if $E_2$ is unbounded, additionally $u,v \in o(V_2)$, and $\{\bfzeta_{z,p}\}_{z \in E, p \in \bR^{q_1+q_2}} \subseteq C(E)$ and $\{\bfxi_{z}\}_{z \in E} \subseteq C^1(E)$ be collections of functions as in Definition \ref{definition:perturbation_first_second_order}. Fix $\eps = (\eps_1, \eps_2)$ with $\eps_1 + \eps_2 \in (0,1)$ and $\varphi \in (0,1)$. 

    \smallskip
        
    Then, there exist compact sets $K_{\varepsilon,0} \subseteq K_\varepsilon \subseteq E$ and, for any $\alpha = (\alpha_1, \alpha_2)$ with $\alpha_1, \alpha_2 > 1$, three pairs of variables $(y_{\alpha,0},y_{\alpha,0}')$, $(y_{\alpha},y_{\alpha}')$, $(x_{\alpha},x_{\alpha}')$ in $E^2$ and
    \begin{align}
        p_{\alpha} \in \big(B_{\alpha^{-1}_1}(0)\cap \{E_1- y_{\alpha,0,1}\}\big)\times \big(B_{\alpha^{-1}_2}(0)\cap \{E_2-y_{\alpha,0,2}\}\big),\\
        p'_{\alpha} \in \big(B_{\alpha^{-1}_1}(0)\cap \{E_1 - y'_{\alpha,0,1}\}\big)\times \big(B_{\alpha^{-1}_2}(0)\cap \{E_2-y'_{\alpha,0,2}\}\big).
    \end{align}    

    \begin{description}
        \item[Properties of $y_{\alpha,0},y_{\alpha,0}'$]
    \end{description} 
    The variables $y_{\alpha,0},y_{\alpha,0}'$ optimize $\ssup{\Lambda_\alpha}$, where
    \begin{multline} \label{eqn:optimizing_points_Lambda}
        \Lambda_\alpha(y,y') \coloneqq\frac{1}{1\boxminus\varepsilon} P^{\alpha}[u](y) - \frac{1}{1\boxplus\varepsilon} P_{\alpha}[v](y') - \alpha \Phi(y,y')  \\
        - \frac{\varepsilon}{1\boxminus\varepsilon} (1-\varphi) \bfV(y) - \frac{\varepsilon}{1\boxplus\varepsilon} (1-\varphi) \bfV(y'),
    \end{multline}
    and satisfy the following property
    \begin{enumerate}[(a)]
        \item \label{item:proposition:optimizing_point_construction_compact0} $y_{\alpha,0},y_{\alpha,0}' \in K_{\varepsilon,0}$.
    \end{enumerate}

    \begin{description}
        \item[Properties of $y_{\alpha},y_{\alpha}'$ and $p_{\alpha},p_{\alpha}'$]
    \end{description} 
    \noindent
    The pair $y_{\alpha},y_{\alpha}'$ optimizes
    \begin{equation} \label{eqn:proposition:optimizing_point_construction:optimizing_points}
        \ssup{ \Lambda_\alpha  - \frac{\varepsilon}{1\boxminus\varepsilon} \varphi \Xi^0_1 - \frac{\varepsilon}{1\boxplus\varepsilon} \varphi \Xi^0_2  }
    \end{equation}
    and uniquely optimizes
    \begin{equation} \label{eqn:proposition:optimizing_point_construction:optimizing_points_unique}
        \ssup{\Lambda_\alpha  - \frac{\varepsilon}{1\boxminus\varepsilon} \varphi \Xi_1 - \frac{\varepsilon}{1\boxplus\varepsilon} \varphi \Xi_2  },
    \end{equation}
    where $\Lambda_\alpha$ is as in \eqref{eqn:optimizing_points_Lambda} and 
    \begin{align*}
        \Xi_1^0(y) &\coloneqq \Xi^0_{y_{\alpha,0}, p_{\alpha}} (y), &
        \Xi_2^0(y') &\coloneqq \Xi^0_{y_{\alpha,0}', p_{\alpha}'} (y'),\\
        \Xi_1(y) &\coloneqq \Xi_{y_{\alpha,0}, p_{\alpha}, y_{\alpha}}(y), &  \Xi_2(y') &\coloneqq \Xi_{y_{\alpha,0}', p_{\alpha}', y_{\alpha}'}(y')
    \end{align*}
    as in Definition \ref{definition:perturbation_first_second_order}.
    Moreover, the optimizers $y_{\alpha},y_{\alpha}'$ of \eqref{eqn:proposition:optimizing_point_construction:optimizing_points} and \eqref{eqn:proposition:optimizing_point_construction:optimizing_points_unique} satisfy
    \begin{enumerate}[(a),resume]
        \item \label{item:proposition:optimizing_point_Jensencontrol} We have
        \begin{equation*}
            d(y_{\alpha},y_{\alpha,0})  \leq \frac{1}{\alpha_1} + \frac{1}{\alpha_2}, \qquad
            d(y_{\alpha}',y_{\alpha,0}')  \leq \frac{1}{\alpha_1} + \frac{1}{\alpha_2}.
        \end{equation*}
        \item \label{item:proposition:optimizing_point_construction_twice_diff} $P^\alpha[u]$ and $P_\alpha[v]$ are twice differentiable in $y_{\alpha}$ and $y_{\alpha}'$, respectively.
    \end{enumerate}

    \begin{description}
        \item[Properties of $x_{\alpha},x_{\alpha}'$]
    \end{description} 
        The variables $x_{\alpha}, x_{\alpha}'$ optimize
    \begin{equation}\label{eqn:construction_optimizing_x}
        \begin{aligned}
            P^\alpha[u](y_{\alpha}) &  = u(x_{\alpha}) - \frac{\alpha}{2} \bfd^2(x_{\alpha}, y_{\alpha}), \\
            P_\alpha[v](y_{\alpha}') & = v(x_{\alpha}') + \frac{\alpha}{2} \bfd^2(x_{\alpha}',y_{\alpha}'),
        \end{aligned}
    \end{equation}
    and satisfy
    \begin{enumerate}[(a),resume]
        \item \label{item:proposition:optimizing_point_construction:Ralpha_u_Ralpha_v_optimizers} $x_{\alpha}$ and $x_{\alpha}'$ are the unique optimizers in the definition of $P^\alpha[u](y_{\alpha})$ and $P_\alpha[v](y_{\alpha}')$, respectively.
        \item \label{item:item:proposition:optimizing_point_construction:shift_optimal}
        We have that
        \begin{align*}
            u(x_{\alpha}) - P^\alpha[u] \circ s_{x_{\alpha}-y_{\alpha}}(x_{\alpha}) & = \ssup{ u - P^\alpha[u] \circ s_{x_{\alpha}-y_{\alpha}} }, \\
            v(x_{\alpha}') - P_\alpha[v] \circ s_{x_{\alpha}'-y_{\alpha}'}(x_{\alpha}') & = \iinf{ v - P_\alpha[v] \circ s_{x_{\alpha}'-y_{\alpha}'} }.
        \end{align*}
    \end{enumerate}

    \begin{description}
        \item[Behaviour as $\alpha_1, \alpha_2 \rightarrow \infty$]
    \end{description} 

    \begin{enumerate}[(a),resume]
        \item \label{item:proposition:optimizing_point_construction_0optimizers_convergence} We have $\lim_{\alpha_1, \alpha_2 \rightarrow \infty} \alpha \Phi(y_{\alpha,0},y_{\alpha,0}') = 0$.
        \item \label{item:proposition:optimizing_point_construction_2optimizers_convergence} 
        We have
        \begin{equation*}
            \lim_{\alpha_1, \alpha_2 \rightarrow \infty} \alpha \Big( \left(\bfd\left(x_{\alpha},y_{\alpha} \right) + \bfd\left(y_{\alpha}',x_{\alpha}' \right) \right)^2 + \Phi(y_\alpha, y'_\alpha) \Big) = 0.
        \end{equation*}
        \item \label{item:proposition:optimizing_point_construction_compact2} $x_{\alpha}, y_{\alpha}, y_{\alpha}', x_{\alpha}' \in K_{\varepsilon}$.
        \end{enumerate}
        In addition, the following estimate on $(u-v)$ holds: For any compact set $K \subseteq E$, there is a compact set $\widehat{K} = \widehat{K}(K,\varepsilon, u,v)$ given by 
        \begin{multline}\label{eqn:definition_hatK}
            \widehat{K} \coloneqq \left\{z \in E \, \middle| \,  \frac{1}{1\boxminus\eps} \left(\eps \bfV(z) - u(z)\right) + \frac{1}{1\boxplus\eps} \left(\eps \bfV(z) - v(z) \right)   \right. \\
            \leq \left. \frac{\eps}{1\boxminus\eps} \ssup{\bfV-u}_K + \frac{\eps}{1\boxplus\eps} \ssup{\bfV-v}_K - \ssup{u-v}_K \right\}
        \end{multline}
        such that
        \begin{enumerate}[(a),resume]
        \item \label{item:proposition:optimizing_point_construction_estimate_u-v}
        For any compact set $K \subseteq E$,
        \begin{align*}
             \ssup{u-v}_K &\leq \frac{1}{1\boxminus\varepsilon}\left(u(x_\alpha) - \varepsilon(1-\varphi) \bfV(y_\alpha)\right) - \frac{1}{1\boxplus\varepsilon}\left(v(x_\alpha') + \varepsilon (1-\varphi) \bfV(y_\alpha')\right)  \\
             &\qquad + \varphi \frac{2\varepsilon}{(1\boxminus\eps)(1\boxplus\eps)} \frac{1}{\alpha} + c^{\eps}_\alpha,
        \end{align*}
        where
        \begin{align} \label{eqn:Cvarphi_construction}
            c^{\eps}_\alpha &\coloneqq  \frac{\varepsilon}{1\boxminus\varepsilon}\ssup{\bfV-P^\alpha [u]}_K + \frac{\varepsilon}{1\boxplus\varepsilon}\ssup{\bfV-P_\alpha [v]}_K
        \end{align}
        and $\varphi \frac{2\varepsilon}{(1\boxminus\eps)(1\boxplus\eps)} \frac{1}{\alpha}$ is $o(1)$ is in terms of $\alpha \rightarrow \infty$ for fixed $\varepsilon$ and $\varphi$.
        \item \label{item:proposition:optimizing_limits}
        Any limit point of the sequence $(x_{\alpha}, y_{\alpha},  y_{\alpha,0}, y_{\alpha,0}', y_{\alpha}', x_{\alpha}')$ as $\alpha \rightarrow \infty$ is of the form $(z,z,z,z,z,z)$ with $z \in \widehat{K}$. 
    \end{enumerate}    
\end{proposition}

\begin{remark}\label{remark:cotangent_bundle}
    To the reader familiar with Differential Geometry, we want to point out that
    \begin{align}
        p_{\alpha} &\in \big(B_{\alpha^{-1}_1}(0)\cap \{E_1- y_{\alpha,0,1}\}\big)\times \big(B_{\alpha^{-1}_2}(0)\cap \{E_2-y_{\alpha,0,2}\}\big),\\
        p'_{\alpha} &\in \big(B_{\alpha^{-1}_1}(0)\cap \{E_1 - y'_{\alpha,0,1}\}\big)\times \big(B_{\alpha^{-1}_2}(0)\cap \{E_2-y'_{\alpha,0,2}\}\big)
    \end{align} 
    should be understood as elements of the cotangent bundle $T^*E$.
\end{remark}

\begin{lemma}\label{lemma:growth_P}
    Let $E$ be unbounded, $V \colon E\to \bR$ be a containment function, and $u \colon E \to \bR$ be such that $u \in o(V)$. Then, for $\alpha > (\kappa_V + 1)$, we have $P^\alpha[u] \in o(V)$.
\end{lemma}

\begin{proof}
    By Definition \ref{definition:perturbation_containment}~\ref{item:definition:perturbation_containment:bounded_below}, $V$ is bounded below. Thus, we might, w.l.o.g., assume that $\iinf{V} = 0$.
    By Definition \ref{definition:perturbation_containment}~\ref{item:definition:perturbation_containment:semi_concave}, $V$ is semi-concave with semi-concavity constant $\kappa_V$. Hence, there exists some constant $C_V>0$ such that have $V(y) \leq C_V + (\kappa_V +1)|y|^2$. In particular, $P^\alpha[V]$ is proper and, for any $y \in E$, its optimizer $x^*$ is attained in the compact set $K_V = \{x \in E \,|\, d(x,y) \leq r'\}$, where $r'>0$.
    Moreover, semi-concavity implies that $V$ is locally Lipschitz continuous.

    Now let $\delta>0$. As $u \in o(V)$, there exits a radius $r>0$ such that, for all $|x|>r$, we have $|u(x)| \leq \delta |V(x)|$.

    Consequently, for any $|y|>r$, we find
    \begin{align}
        P^\alpha[u](y) &= \sup_{x \in E} \left\{u(x) - \frac{\alpha}{2} d^2(x,y) \right\}\\
        &\leq \sup_{x \in E} \left\{\delta V (x) - \frac{\alpha}{2} d^2(x,y) \right\}\\
        &\leq \sup_{x \in E} \left\{\delta (V (y) + L_{K_V}r') - \frac{\alpha}{2} d^2(x,y) \right\}\\
        &= \delta (V (y) + L_{K_V}r'),
    \end{align}
    which implies that $P^\alpha[u] \in o(V)$ as $E$ is unbounded and, by Definition \ref{definition:perturbation_containment}~\ref{item:definition:perturbation_containment:compact_sublevels}, $V$ has compact sublevel sets.
\end{proof}

\begin{proof}
    {\bfseries Proof of \ref{item:proposition:optimizing_point_construction_compact0}:} 
    We start by showing that $\Lambda_\alpha$ as defined in equation \eqref{eqn:optimizing_points_Lambda} is bounded from above. In this proof, w.l.o.g., we assume that $E_1$ is bounded. For $E_2$ we consider the following cases:\\    
    \textit{Case $E_2$ unbounded:} By assumption, we have that $u,v\in o(V_2)$ in $E_2$. 
    By Lemma \ref{lemma:growth_P}, the same holds for $P^\alpha[u]$ and $P_\alpha[v]$ for $\alpha$ large enough.
    Furthermore, by \cite[Lemma 5.2 (d)]{DCFuKrNe24}, the convolutions $P^\alpha[u]$ and $P_\alpha[v]$ are continuous. Using this together with the fact that $E_1$ is bounded, by the Extreme Value Theorem, see, e.g., \cite[Theorem 4.16]{Ru76}, we find that $P^\alpha[u] - \bfV$ and $- P_\alpha[v] - \bfV$ and consequently $\Lambda_\alpha$ are bounded from above.\\
    \textit{Case $E_2$ bounded:} As $P^\alpha[u]$ and $P_\alpha[v]$ are continuous and $E$ is bounded, $\Lambda_\alpha$ is bounded from above.\\
    In any case, $\Lambda_\alpha$ is bounded from above and using that $\bfV$ has compact sublevel sets, cf.\ Definition \ref{definition:perturbation_containment}, the existence of optimizers $(y_{\alpha,0},y_{\alpha,0}')$ of $\ssup{\Lambda_\alpha}$ follows.

    Since $\bfV$ is continuous and has compact sublevel sets, it attains
    its minimum. After replacing $\bfV$ by
    $\bfV-\min_E\bfV$, we may assume that $\min_E\bfV=0$. Fix
    $\hat y\in E$ such that $\bfV(\hat y)=0$.
    
    Comparing the optimizers of $\Lambda_\alpha$ with the suboptimal
    choice $(y,y')=(\hat y,\hat y)$, and again using that $\bfV$ has
    compact sublevel sets, there exist radii
    $R_{\eps_1},R_{\eps_2}>0$ such that 
    \begin{equation}
        y_{\alpha,0} \in \overline{B_{R_{\eps_1}}(\hat{y})} \quad\text{and}\quad y_{\alpha,0}' \in \overline{B_{R_{\eps_2}}(\hat{y})},
    \end{equation}
    and
    \begin{align}
        \frac{\eps}{1\boxminus\eps} \bfV(y_{\alpha}) + \frac{\eps}{1\boxplus\eps} \bfV(y_{\alpha}') &\leq \frac{1}{1\boxminus\eps} \big[P^\alpha[u](y_{\alpha}) - P^\alpha[u](\hat{y})\big]\\
        &\qquad- \frac{1}{1\boxplus\eps} \big[P_\alpha[v](y_{\alpha}') - P_\alpha[v](\hat{y})\big]\\
        &\leq L_{\eps,y_{\alpha}, \hat{y}} \bfd(y_{\alpha}, \hat{y}) + L_{\eps,y_{\alpha}', \hat{y}} \bfd(y_{\alpha}', \hat{y})\\
        &\leq 2 L_\eps (R_{\eps_1}+R_{\eps_2}), \label{eq:opt_prop:initial_V_estimate}
    \end{align}
    where $L_\eps = L_{\eps,y_{\alpha}, \hat{y}} \vee L_{\eps,y_{\alpha}', \hat{y}}$. From this estimate, we deduce that $(y_{\alpha,0},y_{\alpha,0}') \in K_{\varepsilon,0} \times K_{\varepsilon,0}$ with
    \begin{equation*}
        K_{\varepsilon,0} \coloneqq \left\{y \in E \, \middle| \, V(y) \leq \varepsilon^{-1} 2L_\eps (R_{\eps_1}+R_{\eps_2}) \right\}.
    \end{equation*}

    {\bfseries Proof of \ref{item:proposition:optimizing_point_Jensencontrol} and \ref{item:proposition:optimizing_point_construction_twice_diff}:}
    For the proof of these two statements, we first move from $\ssup{\Lambda_\alpha}$ to its perturbed version \eqref{eqn:proposition:optimizing_point_construction:optimizing_points}. To do so, we use Proposition \ref{proposition:Jensen_Alexandrov_cutoff} twice, for $E_1$ and $E_2$ separately. Note that, for every $y_{-i}, y_{-i}' \in E_{-i}$, where $-i \in \{1,2\} \setminus \{i\}$, the function $(y_i,y'_i) \mapsto \Lambda_\alpha(y_i,y'_i; y_{-i}, y_{-i}')$ of \eqref{eqn:optimizing_points_Lambda}, over which we optimize in $\ssup{\Lambda_\alpha}$, by assumption, can be extended to a semi-convex function with semi-convexity constant 
    \begin{equation*}
        \kappa_i = \alpha_i  \left(\frac{1}{1\boxminus\eps} - \frac{1}{1\boxplus\eps} + \kappa_{\phi_i}\right) + \left(\frac{\eps_i}{1\boxminus\eps} - \frac{\eps_i}{1\boxplus\eps}\right) (1-\varphi) \kappa_{V_i} > 1
    \end{equation*}
    for $\alpha_i > 1$ on $\cN(E_i \times E_i)$. In addition, it is bounded from above and has optimizers $(y_{\alpha,0},y'_{\alpha,0}) \in E \times E$. We can thus apply Proposition \ref{proposition:Jensen_Alexandrov_cutoff} with 
    \begin{equation} \label{eqn:optimizer_construct_1}
        \eta_i = \frac{1}{\alpha_i}, \qquad \epsilon_{1,i} = \frac{\varepsilon_i}{1\boxminus\varepsilon} \varphi, \qquad \epsilon_{2,i} = \frac{\varepsilon_i}{1\boxplus\varepsilon} \varphi.
    \end{equation}

    Consequently, it follows that there exist
    \begin{align}
        p_{\alpha} \in \big(B_{\alpha^{-1}_1}(0)\cap \{E_1- y_{\alpha,0,1}\}\big)\times \big(B_{\alpha^{-1}_2}(0)\cap \{E_2-y_{\alpha,0,2}\}\big),\\
        p'_{\alpha} \in \big(B_{\alpha^{-1}_1}(0)\cap \{E_1 - y'_{\alpha,0,1}\}\big)\times \big(B_{\alpha^{-1}_2}(0)\cap \{E_2-y'_{\alpha,0,2}\}\big),
    \end{align}
    such that $y_{\alpha},y_{\alpha}'$ are optimizers of
    \begin{equation}\label{eqn:constructionOptimizersHatPhi}
        \ssup{\widehat{\Lambda}_\alpha} = \widehat{\Lambda}_\alpha(y_{\alpha},y_{\alpha}'),
    \end{equation}
    where
    \begin{equation} \label{eqn:constructionOptimizersHatPhi2}
    \widehat{\Lambda}_\alpha(y,y') := \Lambda_\alpha(y,y')  - \frac{\varepsilon}{1\boxminus\varepsilon} \varphi \Xi_1^0(y) - \frac{\varepsilon}{1\boxplus\varepsilon} \varphi \Xi_2^0(y')
    \end{equation}
    with $\Xi^0_1$ and $\Xi^0_2$ as defined above.
    This establishes \eqref{eqn:proposition:optimizing_point_construction:optimizing_points}. An additional penalization around $(y_{\alpha},y_{\alpha}')$ then gives \eqref{eqn:proposition:optimizing_point_construction:optimizing_points_unique}. 
    A secondary outcome of Proposition \ref{proposition:Jensen_Alexandrov_cutoff} is that $\widehat{\Lambda}_\alpha$ is twice differentiable in the optimizing point $(y_{\alpha}, y_{\alpha}')$, establishing \ref{item:proposition:optimizing_point_construction_twice_diff}. Furthermore, the optimizers satisfy
    \begin{equation} \label{eqn:optimizer_construct_2}
    \bfd(y_{\alpha},y_{\alpha,0}) < \eta_1 + \eta_2, \qquad \bfd(y_{\alpha}',y_{\alpha,0}') < \eta_1 + \eta_2,
    \end{equation}
    which, together with \eqref{eqn:optimizer_construct_1}, yields
    \begin{equation}\label{eqn:optimizer_construct_3}
        \max \left\{ \bfd(y_{\alpha},y_{\alpha,0}), \bfd(y_{\alpha}',y_{\alpha,0}')\right\} \leq \frac{1}{\alpha_1} + \frac{1}{\alpha_2},
    \end{equation}
    establishing \ref{item:proposition:optimizing_point_Jensencontrol}. 

    {\bfseries Proof of \ref{item:proposition:optimizing_point_construction:Ralpha_u_Ralpha_v_optimizers}:}
    This follows immediately from \cite[Lemma 5.2 (e)]{DCFuKrNe24}. 

    {\bfseries Proof of \ref{item:item:proposition:optimizing_point_construction:shift_optimal}:}
    We only establish
    \begin{equation*}
        u(x_{\alpha}) - P^\alpha[u] \circ s_{x_{\alpha}-y_{\alpha}}(x_{\alpha}) = \ssup{ u - P^\alpha[u] \circ s_{x_{\alpha}-y_{\alpha}}},
    \end{equation*}
    as the second equation follows similarly. Note that, by definition of $P^\alpha[u]$, for all 
    $x \in E$, we have
    \begin{equation*}
        P^\alpha[u] \circ s_{x_{\alpha}-y_{\alpha}}(x) \geq u(x) - \frac{\alpha}{2} \bfd^2\big(x, s_{x_{\alpha} - y_{\alpha}}(x)\big).
    \end{equation*}
    On the other hand, by \ref{item:proposition:optimizing_point_construction:Ralpha_u_Ralpha_v_optimizers}, we have
    \begin{equation*}
        P^\alpha[u] \circ s_{x_{\alpha}-y_{\alpha}}(x_{\alpha}) = P^\alpha[u](y_{\alpha}) = u(x_{\alpha}) - \frac{\alpha}{2} \bfd^2(x_{\alpha}, y_{\alpha}).
    \end{equation*}
    Combining the two statements yields, for any 
    $x \in E$, that
    \begin{align*}
        & u(x_{\alpha}) - P^\alpha[u] \circ s_{x_{\alpha}-y_{\alpha}}(x_{\alpha}) \\
        & \qquad = \frac{\alpha}{2} \bfd^2\left(x_{\alpha}, y_{\alpha} \right) + P^\alpha[u] \circ s_{x_{\alpha}-y_{\alpha}}(x) - P^\alpha[u] \circ s_{x_{\alpha}-y_{\alpha}}(x) \\
        & \qquad \geq u(x) - P^\alpha[u] \circ s_{x_{\alpha}-y_{\alpha}}(x) + \frac{\alpha}{2}\big[\bfd^2\left(x_{\alpha}, y_{\alpha} \right) - \bfd^2\left(x, s_{x_{\alpha} - y_{\alpha}}(x)\right)\big] \\
        & \qquad =  u(x) - P^\alpha[u] \circ s_{x_{\alpha}-y_{\alpha}}(x) 
    \end{align*}
    as the shift map preserves distances. 
    This establishes \ref{item:item:proposition:optimizing_point_construction:shift_optimal}.

    \smallskip

    For the proof of the final five properties, we consider the limit $\alpha_1, \alpha_2 \rightarrow \infty$. Note that even though all statements below are for $\alpha_1, \alpha_2 \rightarrow \infty$, analogous statements hold for fixed $\alpha_2$ as $\alpha_1 \rightarrow \infty$ and fixed $\alpha_1$ as $\alpha_2 \rightarrow \infty$. Thus, the iterated and double limits converge to the same value.
    
    \smallskip

    {\bfseries Proof of \ref{item:proposition:optimizing_point_construction_0optimizers_convergence}:}
    Consider $ \ssup{\Lambda_\alpha}$:
    \begin{align*}
        \ssup{\Lambda_\alpha} & = \frac{1}{1\boxminus\varepsilon} P^\alpha[u](y_{\alpha,0}) - \frac{1}{1\boxplus\varepsilon} P_\alpha[v](y'_{\alpha,0}) - \alpha \Phi(y_{\alpha,0},y'_{\alpha,0}) \\
        & \qquad - \frac{\varepsilon}{1\boxminus\varepsilon} (1-\varphi) \bfV(y_{\alpha,0}) - \frac{\varepsilon}{1\boxplus\varepsilon} (1-\varphi) \bfV(y_{\alpha,0}').
    \end{align*}
    Note, that $\ssup{\Lambda_\alpha}$ is decreasing in $\alpha_1$ and $\alpha_2$, since $-\frac{\alpha}{2} \Phi$, $P^{\alpha}[u]$, and $- P_\alpha [v]$ are decreasing in $\alpha$ by \cite[Lemma 5.2 (c)]{DCFuKrNe24}. 
    Note in addition that, by evaluating $\Lambda_\alpha$ in the particular choice $(y,y') = (\hat{y}, \hat{y})$ as above, we have, by \cite[Lemma 5.2 (a)]{DCFuKrNe24}, that
    \begin{equation*}
        \ssup{\Lambda_\alpha}  \geq \frac{1}{1\boxminus\varepsilon} P^\alpha[u](\hat{y}) - \frac{1}{1\boxplus\varepsilon} P_\alpha[v](\hat{y}) \geq \frac{1}{1\boxminus\varepsilon} u(\hat{y}) - \frac{1}{1\boxplus\varepsilon} v(\hat{y}),
    \end{equation*}
    which is lower bounded uniformly in both $\alpha_1$ and $\alpha_2$. It follows that $\lim_{\alpha_1, \alpha_2 \to \infty} \ssup{\Lambda_\alpha} $ exists.

    For any $\alpha_1, \alpha_2 > 1$, we find    
    \begin{align}
         \ssup{\Lambda_{\alpha/2}}  &\geq \frac{1}{1\boxminus\varepsilon} P^{\alpha/2}[u](y_{\alpha,0}) - \frac{1}{1\boxplus\varepsilon} P_{\alpha/2}[v](y'_{\alpha,0}) - \frac{\alpha}{2}\Phi(y_{\alpha,0},y'_{\alpha,0}) \nonumber \\
        &\qquad - \frac{\varepsilon}{1\boxminus\varepsilon}(1-\varphi)\bfV(y_{\alpha,0}) - \frac{\varepsilon}{1\boxplus\varepsilon} (1-\varphi) \bfV(y_{\alpha,0}') \nonumber \\
        &\geq \ssup{\Lambda_\alpha} + \frac{\alpha}{2}\Phi(y_{\alpha,0},y_{\alpha,0}'),
        \label{eqn:Malpha_bound}
    \end{align}
    which implies that $\lim_{\alpha_1, \alpha_2 \to \infty}\alpha \Phi(y_{\alpha,0},y_{\alpha,0}')=0$, as $\ssup{\Lambda_\alpha}$ and $\ssup{\Lambda_{\alpha/2}}$ converge to the same limit, establishing \ref{item:proposition:optimizing_point_construction_0optimizers_convergence}.

    {\bfseries Proof of \ref{item:proposition:optimizing_point_construction_2optimizers_convergence}:}
    We follow the same approach as in \eqref{eqn:Malpha_bound} but now expanding $P^\alpha[u](y_{\alpha})$ and $P_\alpha[v](y_{\alpha}')$ to obtain an optimization problem in terms of four variables: By \eqref{eqn:constructionOptimizersHatPhi2}, we have
    \begin{equation}
        \begin{aligned}
            \ssup{\Lambda_{\alpha/2}} &\geq \frac{1}{1\boxminus\varepsilon} P^{\alpha/2}[u](y_{\alpha}) - \frac{1}{1\boxplus\varepsilon} P_{\alpha/2}[v](y'_{\alpha}) - \frac{\alpha}{2}\Phi(y_{\alpha},y'_{\alpha}) \\
            & \qquad - \frac{\varepsilon}{1\boxminus\varepsilon}(1-\varphi)\bfV(y_{\alpha}) - \frac{\varepsilon}{1\boxplus\varepsilon}(1-\varphi) \bfV(y_{\alpha}')\\
            &  \geq \ssup{\widehat{\Lambda}_{\alpha}} + \frac{\alpha}{4} \left(\frac{1}{1\boxminus\varepsilon}\bfd^2(x_{\alpha},y_{\alpha}) + 2\Phi(y_{\alpha},y_{\alpha}') + \frac{1}{1\boxplus\varepsilon} \bfd^2(y_{\alpha}',x_{\alpha}') \right) \\
            & \qquad + \frac{\eps}{1\boxminus\eps} \varphi \Xi_1^0(y_{\alpha}) + \frac{\eps}{1\boxplus\eps} \varphi \Xi_2^0(y'_{\alpha}).
        \end{aligned}
    \end{equation}
    It follows that 
    \begin{multline}
        \frac{\alpha}{4} \left(\frac{1}{1\boxminus\varepsilon}\bfd^2(x_{\alpha},y_{\alpha}) + 2\Phi(y_{\alpha},y_{\alpha}') + \frac{1}{1\boxplus\varepsilon} \bfd^2(y_{\alpha}',x_{\alpha}') \right)\\ 
        \leq \ssup{\Lambda_{\alpha/2}} - \ssup{\widehat{\Lambda}_{\alpha}} - \frac{\eps}{1\boxminus\eps} \varphi \Xi_1^0(y_{\alpha}) - \frac{\eps}{1\boxplus\eps} \varphi \Xi_2^0(y'_{\alpha}).
    \end{multline}
    Note that plugging the definition of $\eta$ as in equation \eqref{eqn:optimizer_construct_1} into the results of Corollary \ref{corollary:Jensen_optimizerbound} yields
    \begin{equation}\label{eqn:Xi_optimizer_nonnegative}
    \begin{split}
    0 \leq - \frac{\varepsilon}{1\boxminus\varepsilon} \varphi \Xi_1^0(y_{\alpha})  - \frac{\varepsilon}{1\boxplus\varepsilon} \varphi \Xi_2^0(y_{\alpha}') \leq \varphi \left(\frac{2\varepsilon_1}{(1\boxminus\eps)(1\boxplus\eps)} \frac{1}{\alpha_1} + \frac{2\varepsilon_2}{(1\boxminus\eps)(1\boxplus\eps)} \frac{1}{\alpha_2}\right)
    \end{split}
    \end{equation}
    and
    \begin{equation} \label{eqn:Jensen_control_optimizationproblem_inOptimizerConstruction}
        \ssup{\Lambda_\alpha} \leq \ssup{\widehat{\Lambda}_\alpha}  = \widehat{\Lambda}_\alpha(y_{\alpha},y_{\alpha}') 
        \leq \ssup{\Lambda_\alpha} + \varphi \frac{2\varepsilon}{(1\boxminus\eps)(1\boxplus\eps)} \frac{1}{\alpha}.
    \end{equation}
    Consequently, by \eqref{eqn:Xi_optimizer_nonnegative}, we have
    \begin{equation} 
        \lim_{\alpha_1,\alpha_2 \rightarrow \infty} \frac{\varepsilon}{1\boxminus\varepsilon} \varphi \Xi_{1}^0(y_{\alpha}) + \frac{\varepsilon}{1\boxplus\varepsilon} \varphi \Xi_{2}^0(y_{\alpha}') = 0
    \end{equation}
    and, combining \eqref{eqn:Jensen_control_optimizationproblem_inOptimizerConstruction} with \ref{item:proposition:optimizing_point_construction_0optimizers_convergence}, we obtain
    \begin{equation} 
        \lim_{\alpha_1,\alpha_2 \rightarrow \infty} \ssup{\Lambda_\alpha} = \lim_{\alpha_1,\alpha_2 \rightarrow \infty} \ssup{\widehat{\Lambda}_\alpha}.
    \end{equation}
    Thus, we find
    \begin{equation*}
        \lim_{\alpha \rightarrow \infty} \alpha \left(\bfd^2(x_{\alpha},y_{\alpha}) + \Phi(y_{\alpha},y_{\alpha}') + \bfd^2(y_{\alpha}',x_{\alpha}') \right) = 0.
    \end{equation*}
    From this, \ref{item:proposition:optimizing_point_construction_2optimizers_convergence} follows using Young's inequality.

    {\bfseries Proof of \ref{item:proposition:optimizing_point_construction_compact2}:}
    \ref{item:proposition:optimizing_point_construction_compact0}, \ref{item:proposition:optimizing_point_Jensencontrol}, \ref{item:proposition:optimizing_point_construction_0optimizers_convergence}, and \ref{item:proposition:optimizing_point_construction_2optimizers_convergence} imply \ref{item:proposition:optimizing_point_construction_compact2} by considering a bounded blow-up $K_{\varepsilon}$ of $K_{\varepsilon,0}$ that is contained in $E$.

    Before we continue with the rest of the proof, we point out that
    \begin{align}
        1\boxplus\eps = 1 + \eps_1 + \eps_2 \quad\text{and}\quad
        1\boxminus\eps = 1 - \eps_1 - \eps_2 .
    \end{align}

    {\bfseries Proof of \ref{item:proposition:optimizing_point_construction_estimate_u-v}:}
    Let $K \subseteq E$ be compact. Set
    \begin{equation}\label{eq:def_c_K_alpha}
        c_{K, \alpha}^u \coloneqq \ssup{\bfV -P^\alpha [u]}_K \quad\text{and}\quad c_{K, \alpha}^v \coloneqq \ssup{\bfV- P_\alpha [v]}_K.
    \end{equation}
    Note that, for any $x\in E$, we have the following identity:
    \begin{align}
        &P^\alpha [u](x) - P_\alpha [v](x)\\
        &\quad = \frac{1}{1\boxminus\varepsilon}\left(P^\alpha [u](x) - \varepsilon (1-\varphi)\bfV(x)\right) - \frac{1}{1\boxplus\varepsilon}\left(P_\alpha [v](x) + \varepsilon (1-\varphi)\bfV(x)\right) \\
        & \qquad - \frac{\varepsilon}{1\boxminus\varepsilon}(P^\alpha [u](x) -(1-\varphi)\bfV(x)) - \frac{\varepsilon}{1\boxplus\varepsilon}(P_\alpha [v](x) -(1-\varphi)\bfV(x))
        \label{eq:u-v_identity}
    \end{align}
    
    Consequently, we can estimate
    \begin{align}
        &\ssup{u-v}_K \leq \ssup{P^\alpha [u] - P_\alpha [v]}_K = \sup_{x \in K} P^\alpha [u](x) - P_\alpha [v](x) \\
        &\quad = \sup_{x \in K} \frac{1}{1\boxminus\varepsilon}\left(P^\alpha [u](x) - \varepsilon (1-\varphi)\bfV(x)\right) - \frac{1}{1\boxplus\varepsilon}\left(P_\alpha [v](x) + \varepsilon (1-\varphi)\bfV(x)\right) \\
        & \qquad - \frac{\varepsilon}{1\boxminus\varepsilon}(P^\alpha [u](x) -(1-\varphi)\bfV(x)) - \frac{\varepsilon}{1\boxplus\varepsilon}(P_\alpha [v](x) -(1-\varphi)\bfV(x))\\
        &\quad \leq \sup_{x \in K} \frac{1}{1\boxminus\varepsilon}\left(P^\alpha [u](x) - \varepsilon (1-\varphi)\bfV(x)\right) - \frac{1}{1\boxplus\varepsilon}\left(P_\alpha [v](x) + \varepsilon (1-\varphi)\bfV(x)\right) \\
        & \qquad + \frac{\varepsilon}{1\boxminus\varepsilon}c_{K, \alpha}^u + \frac{\varepsilon}{1\boxplus\varepsilon}c_{K, \alpha}^v\\
        &\quad \leq \sup_{x \in E} \frac{1}{1\boxminus\varepsilon} \left(P^\alpha [u](x) - \varepsilon (1-\varphi)\bfV(x) \right)  - \frac{1}{1\boxplus\varepsilon} \left( P_\alpha [v](x) + \varepsilon (1-\varphi)\bfV(x)\right) \\
        &\qquad  + \frac{\varepsilon}{1\boxminus\varepsilon}c_{K, \alpha}^u + \frac{\varepsilon}{1\boxplus\varepsilon}c_{K, \alpha}^v\\
        &\quad \leq \sup_{y,y' \in E} \frac{1}{1\boxminus\varepsilon} \left(P^\alpha [u](y) - \varepsilon (1-\varphi) \bfV(y) \right) - \frac{1}{1\boxplus\varepsilon} \left( P_\alpha[v](y') + \varepsilon (1-\varphi) \bfV(y')\right)  \\
        &\qquad - \alpha \Phi(y,y') + \frac{\varepsilon}{1\boxminus\varepsilon}c_{K, \alpha}^u + \frac{\varepsilon}{1\boxplus\varepsilon}c_{K, \alpha}^v\\
        &\quad = \ssup{\Lambda_\alpha} + \frac{\varepsilon}{1\boxminus\varepsilon}c_{K, \alpha}^u + \frac{\varepsilon}{1\boxplus\varepsilon}c_{K, \alpha}^v \label{eq:improved_u-v_estimate}
    \end{align}

    Combining the above estimate with $\ssup{\Lambda_\alpha} \leq \ssup{\widehat{\Lambda}_\alpha}  = \widehat{\Lambda}_\alpha(y_{\alpha},y_{\alpha}')$, see equation \eqref{eqn:Jensen_control_optimizationproblem_inOptimizerConstruction}, then dropping the non-positive terms and using inequality \eqref{eqn:Xi_optimizer_nonnegative}, leads to
    \begin{align}\label{eq:outcome_i}
         &\ssup{u-v}_K  
         \leq \frac{1}{1\boxminus\varepsilon}\big(u(x_\alpha) - \varepsilon(1-\varphi) \bfV(y_\alpha)\big) - \frac{1}{1\boxplus\varepsilon}\big(v(x_\alpha') + \varepsilon (1-\varphi) \bfV(y_\alpha')\big) \\
         &\qquad + \varphi \frac{2\varepsilon}{(1\boxminus\eps)(1\boxplus\eps)} \frac{1}{\alpha} + \frac{\varepsilon}{1\boxminus\varepsilon}c_{K,\alpha}^u + \frac{\varepsilon}{1\boxplus\varepsilon}c_{K,\alpha}^v,
    \end{align}
    which proves \ref{item:proposition:optimizing_point_construction_estimate_u-v}.

    {\bfseries Proof of \ref{item:proposition:optimizing_limits}:}
    We start by proving that any limiting point of 
    \begin{equation*}(x_{\alpha},y_{\alpha},y_{\alpha,0},y_{\alpha,0}',y_{\alpha}',x_{\alpha}')
    \end{equation*}
    as $\alpha_1, \alpha_2 \rightarrow \infty$ is of the form $(z,z,z,z,z,z)$. We only prove $\lim_{\alpha_1, \alpha_2 \rightarrow \infty} \bfd(x_{\alpha},y_{\alpha}) = 0$, as the other limits follow analogously. 
    
    By \ref{item:proposition:optimizing_point_construction_compact2}, we find that, along subsequences, $(x_{\alpha},y_{\alpha}) \rightarrow (x_0,y_0)$. Assume by contradiction that $x_0 \neq y_0$. Then, since $\alpha \bfd^2$ is increasing in $\alpha$, we get that, for all $\alpha_0 >1$,
    \begin{equation}
        \liminf_{\alpha_1,\alpha_2 \to \infty} \alpha \bfd^2(x_{\alpha}, y_{\alpha}) \geq \alpha_0 \bfd^2(x_0,y_0).
    \end{equation}
    We can conclude that $\alpha \bfd^2(x_{\alpha},y_{\alpha}) \to \infty$, contradicting \ref{item:proposition:optimizing_point_construction_2optimizers_convergence}.
    Note that the same argument also works for the $\Phi$ term, as $\Phi$ separates points.
    
    We proceed to prove that any limiting point $z$ lies in $\widehat{K}$. 
    Considering the outcome of \ref{item:proposition:optimizing_point_construction_estimate_u-v} in equation \eqref{eq:outcome_i}, taking the limit for $\alpha_1, \alpha_2 \to \infty$ and $\varphi\downarrow 0$, and rearranging, \ref{item:proposition:optimizing_limits} follows.
\end{proof}

\begin{remark}\label{remark:c_K_convergence}
    Note that due to the properties of the $\sup$- and $\inf$-convolutions, cf.\ \cite[Theorem 3.5.8]{CaSi04}, we have that
    \begin{align}
        c_{K, \alpha}^u = \ssup{\bfV - P^\alpha [u]}_K &\xrightarrow{\alpha \to \infty} \ssup{\bfV-u}_K = c_K^u,\\
        c_{K, \alpha}^v = \ssup{\bfV-P_\alpha [v]}_K &\xrightarrow{\alpha \to \infty}\ssup{\bfV-v}_K = c_K^v.
    \end{align}
\end{remark}

\subsection{Test function construction}

The proposition in this section builds upon the constructed optimizers from Proposition \ref{prop:optimizer_construction} to build suitable test functions. Note that the sup- and inf-convolution, $P^\alpha[u]$ and $P_\alpha[v]$, are not guaranteed to be smooth but, as we showed above, their second derivatives exist in the relevant optimizing points.

Using the difference between $\Xi_1^0$ and $\Xi_2^0$ on one hand and $\Xi_1$ and $\Xi_2$ on the other, we can use the Whitney Extension Theorem \cite[Theorem 2.3.6]{MR1996773} to find globally $C^\infty$ functions, $\widehat{f}_\dagger$ and $\widehat{f}_\ddagger$, that are squeezed in between 
and can be used to replace $P^\alpha[u]$ and $P_\alpha[v]$ in the comparison proof. 
In the subsolution case, for example, we start by first constructing $\widehat{f}_1 \in C_c^\infty(\domain)$, which, by re-arrangement, satisfies
\begin{equation*}
    \widehat{f}_1 (y) \approx \frac{1}{1\boxminus\varepsilon} P^\alpha[u](y) - \frac{\varepsilon}{1\boxminus\varepsilon} (1-\varphi) \bfV(y)  - \frac{\varepsilon}{1\boxminus\varepsilon} \varphi \Xi_1(y)
\end{equation*}
around the constructed optimizers and is constant outside of a compact set. As $\bfV$ has compact sublevel sets and other terms on the right-hand side are bounded from above, it suffices to first perform a smooth approximation and cut off the result. 

Recall that $1\boxminus\eps = 1-\eps_1-\eps_2$ and $1\boxplus \eps = 1+\eps_1+\eps_2$.

\begin{proposition}[Test function construction] \label{proposition:test_function_construction}
    Consider the setting of Proposition \ref{prop:optimizer_construction}. Fix $\eps = (\eps_1, \eps_2)$ with $\eps_1 + \eps_2 \in (0,1)$, $\varphi \in (0, 1)$, and $\alpha = (\alpha_1, \alpha_2)$ with $\alpha_1, \alpha_2 > 1$. Then, there are functions $f_1,f_2, \widehat{f}_1,\widehat{f}_2 \in C_c^\infty(\domain)$ such that
    \begin{align*}
        f_1 = \widehat{f}_1 \circ s_{x_{\alpha}-y_{\alpha}}, \qquad f_2 = \widehat{f}_2 \circ s_{x_{\alpha}'-y_{\alpha}'},\\
        f_\dagger = \widehat{f}_\dagger \circ s_{x_{\alpha}-y_{\alpha}}, \qquad f_\ddagger = \widehat{f}_\ddagger \circ s_{x_{\alpha}'-y_{\alpha}'},
    \end{align*}
    with
    \begin{align*}
        \widehat{f}_\dagger &=  (1\boxminus\eps) \widehat{f}_1 + (1-\varphi)\eps \bfV + \varphi\eps \Xi, \\
        \widehat{f}_\ddagger &=  (1\boxplus\eps) \widehat{f}_2 - (1-\varphi)\eps\bfV - \varphi\eps\Xi.
    \end{align*}
    satisfying the following properties:

    \noindent For $\widehat{f}_1,\widehat{f}_2$ and $f_1,f_2$, we have
    \begin{enumerate}[(a)]
        \item \label{item:test_function_construction:optimizing_f1f2} The pair $(y_{\alpha},y_{\alpha}')$ is the unique optimizing pair of
        \begin{equation*}
            \widehat{f}_1(y_{\alpha}) - \widehat{f}_2(y_{\alpha}') - \alpha \Phi(y_{\alpha},y_{\alpha}') = \ssup{ \widehat{f}_1 -\widehat{f}_2 - \alpha \Phi }
        \end{equation*}
        and the pair $(x_{\alpha},x_{\alpha}')$ is the unique optimizing pair of
        \begin{equation}
            f_{1}(x_{\alpha}) - f_{2}(x_{\alpha}') - \alpha \Phi_{x_{\alpha}-y_{\alpha},\; x_{\alpha}'-y_{\alpha}'}(x_{\alpha},x_{\alpha}')
            = \ssup{f_{1} - f_{2} -\alpha \Phi_{x_{\alpha}-y_{\alpha},\; x_{\alpha}'-y_{\alpha}'} }.
        \end{equation}   
    \end{enumerate}
    For $\widehat{f}_\dagger,\widehat{f}_\ddagger$, and $f_\dagger,f_\ddagger$ we have
    \begin{enumerate}[(a),resume]       
        \item \label{item:test_function_construction:f1f2_squeeze_Ralpha_u_Ralpha_v} We have
        \begin{equation} \label{eqn:bounds_for_Ralpha}
            P^\alpha[u](y)  \leq \widehat{f}_{\dagger}(y) \quad\text{and}\quad
                P_\alpha[v](y')  \geq \widehat{f}_\ddagger(y')
        \end{equation}
        with equality in $y_{\alpha}$ and $y_{\alpha}'$, respectively.
       \item \label{item:test_function_construction:fdagger_f_ddagger_test_functions_uv} We have that $x_{\alpha}, x'_{\alpha}$ are the unique points such that
       \begin{equation}
           u(x_{\alpha}) - f_\dagger(x_{\alpha}) = \ssup{ u - f_\dagger } \quad\text{and}\quad v(x_{\alpha}') - f_\ddagger(x_{\alpha}')  = \iinf{ v - f_\ddagger }.
       \end{equation}
        \item \label{item:test_function_construction:f1f2_equal_Ralpha_u_Ralpha_v_derivatives} We have
        \begin{equation}\label{eqn:bounds_for_Ralpha_derivatives}
            \begin{aligned} 
                D \widehat{f}_\dagger(y_{\alpha}) & = D f_\dagger (x_{\alpha}) =  \alpha(y_{\alpha}-x_{\alpha}), 
                &D^2 \widehat{f}_\dagger(y_{\alpha}) = D^2 f_\dagger(x_{\alpha}), \\
                D \widehat{f}_\ddagger(y_{\alpha}') & = D f_\ddagger (x_{\alpha}')  = \alpha(x_{\alpha}'-y_{\alpha}'), \
                &D^2 \widehat{f}_\ddagger(y_{\alpha}')  = D^2f_\ddagger(x_{\alpha}').
            \end{aligned}
        \end{equation}
    \end{enumerate}
\end{proposition}

As the proof of the above Proposition is analogous to the proof of \cite[Proposition 5.3]{DCFuKrNe24}, we only provide a short sketch here.
\begin{proof}[Sketch of the proof of Proposition \ref{proposition:test_function_construction}]
Using \cite[Lemma B.1]{DCFuKrNe24}, we find smooth functions $\mathfrak{f}_1, \mathfrak{f}_2 \in C^\infty(E)$ that are squeezed between $\frac{1}{1\boxminus\eps}P^\alpha[u] - \frac{\eps}{1-\eps}(1\boxminus\varphi)\bfV - \frac{\eps}{1\boxminus\eps}\Xi_1^0$ and $\frac{1}{1\boxminus\eps}P^\alpha[u] - \frac{\eps}{1\boxminus\eps}(1-\varphi)\bfV - \frac{\eps}{1\boxminus\eps}\Xi_1$ as well as their analogues for $v$. By construction, the same optimizers as in Proposition \ref{prop:optimizer_construction} also optimize $\ssup{\mathfrak{f}_1 - \mathfrak{f}_2 - \alpha \Phi}$. As we require our test functions to be in $C_c^\infty(E)$, we cut $\mathfrak{f}_1, \mathfrak{f}_2$ off in a suitable way and obtain $\widehat{f}_1, \widehat{f}_2$ such that
\begin{align}
    \widehat{f}_1 (y) &\approx \frac{1}{1\boxminus\varepsilon} P^\alpha[u](y) - \frac{\varepsilon}{1\boxminus\varepsilon} (1-\varphi) \bfV(y)  - \frac{\varepsilon}{1\boxminus\varepsilon} \varphi \Xi_1(y),\\
    \widehat{f}_2 (y') &\approx \frac{1}{1\boxplus\varepsilon} P_\alpha[v](y') + \frac{\varepsilon}{1\boxplus\varepsilon} (1-\varphi) \bfV(y')  + \frac{\varepsilon}{1\boxplus\varepsilon} \varphi \Xi_2(y')
\end{align}
holds around the constructed optimizers and they are constant outside of a compact set. Preparing for a later convexity estimate, we then set
\begin{align*}
    \widehat{f}_\dagger &=  (1\boxminus\eps) \widehat{f}_1 + (1-\varphi)\eps \bfV + \varphi\eps \Xi, \\
    \widehat{f}_\ddagger &=  (1\boxplus\eps) \widehat{f}_2 - (1-\varphi)\eps\bfV - \varphi\eps\Xi.
\end{align*}
Using the results from Proposition \ref{prop:optimizer_construction} and definitions of the test functions $\widehat{f}_\dagger$ and $\widehat{f}_\ddagger$ as well as their shifted versions $f_\dagger$ and $f_\ddagger$, the rest of the statements follow.
\end{proof}

\section{Proof of the strict comparison principle}\label{section:main_proof}
In this section, we utilize the optimizer and test function construction in Section \ref{sec:opt_construction} to prove the main result.
Section \ref{subsection:proof_comparison_toDiffH} connects sub- and supersolutions of $-\bH_1$ and $-\bH_2$ to those of $- H_+$ and $- H_-$, respectively, and contains a key proposition that states that the strict comparison principle holds if an estimate on the difference of Hamiltonians holds.
Section \ref{sec:proof_main_thm} then uses the results from the previous subsection for the proof of the strict comparison principle, Theorem \ref{th:comparison_HJI}.

\subsection{Comparison from an estimate on the Hamiltonians}\label{subsection:proof_comparison_toDiffH}
\begin{remark}
    We point out that if $\bH_1$ and $\bH_2$ satisfy Assumption \ref{assumption:domain_setup}, by definition, any subsolution to $ \bH_1f = 0$ is also a viscosity subsolution to $H_+ f = 0$ and any supersolution to $\bH_2f = 0$ is also a viscosity supersolution to $H_- f = 0$. In a setting with bounded test functions this argument requires the use of sequential denseness, cf.\ \cite[Lemma 6.1]{DCFuKrNe24}.
\end{remark}

\begin{lemma} \label{lemma:test_functions_in_domain}
    Let $\bH_1 \subseteq C(E) \times \big(C(\interior{E}) \cap \mathrm{LSC}(E)\big)$ and $\bH_2 \subseteq C(E) \times \big(C(\interior{E}) \cap \mathrm{USC}(E)\big)$ be operators satisfying Assumptions \ref{assumption:domain_setup} and \ref{assumption:compatibility}. 
    Let $\widehat{f_\dagger}, f_\dagger$ and $\widehat{f_\ddagger}, f_\ddagger$ be as in Theorem \ref{proposition:test_function_construction}.
    Then, $\widehat{f_\dagger}, f_\dagger \in \cD(H_+)$ and $\widehat{f_\ddagger}, f_\ddagger \in \cD(H_-)$.
\end{lemma}
\begin{proof}
    See \cite[Lemma 6.2]{DCFuKrNe24}.
\end{proof}

Note that, w.l.o.g., we are still considering the setup in which $E_1$ is bounded and thus has a boundary. For $E_2$ though, we make case distinctions. Basic examples of such cases are parabolic problems, for which the space component can have a boundary or not, i.e., parabolic equations on $[0, T] \times \bR^q$ or $[0,T] \times [-R, R]$. For a more detailed treatment of such examples we refer to Section \ref{sec:parabolic_example}.

\begin{proposition}\label{proposition:basic_comparison_using_Hestimate}
    Let $\bH_1 \subseteq C(E) \times \big(C(\interior{E}) \cap \mathrm{LSC}(E)\big)$ and $\bH_2 \subseteq C(E) \times \big(C(\interior{E}) \cap \mathrm{USC}(E)\big)$ satisfy Assumptions \ref{assumption:domain_setup} and \ref{assumption:compatibility}. Consider the equations
    \begin{align}
        -H_+f &\leq 0,\label{eq:main_step:subsol}\\
        -H_-f &\geq 0.\label{eq:main_step:supersol}
    \end{align}
    Let $u$ and $v$ be viscosity sub- and supersolutions to \eqref{eq:main_step:subsol} and \eqref{eq:main_step:supersol}, respectively. For $\eps = (\eps_1, \eps_2)$ with $\eps_1 + \eps_2 \in (0,1)$, $\varphi \in (0,1)$, and $\alpha = (\alpha_1, \alpha_2)$ with $\alpha_1, \alpha_2 > 1$, consider the optimizers $(x_\alpha, x_\alpha')$ and test functions $f_\dagger, f_\ddagger$, as constructed in Propositions \ref{prop:optimizer_construction} and \ref{proposition:test_function_construction}.

    Suppose there exist maps $\eps_1 \mapsto C^0_{\eps_1}$ and $\eps_2 \mapsto C^0_{\eps_2}$ such that, for $i \in \{1,2\}$, if $\partial E_i \neq \emptyset$, we have $-\infty<\limsup_{\eps_i \downarrow 0} C^0_{\eps_i} < 0$ and, if $\partial E_i = \emptyset$, $0\leq \limsup_{\eps_i \downarrow 0} C^0_{\eps_i} < \infty$, and a non-negative map $\varphi \mapsto C_{\eps, \varphi}$ with $\lim_{\varphi \downarrow 0} C_{\eps, \varphi} = 0$ such that
    \begin{equation}\label{eq:main_step:H-H_estimate}
        \liminf_{\alpha\rightarrow\infty} \frac{\bH f_\dagger (x_\alpha)}{(1-\eps_1-\eps_2)} - \frac{\bH f_\ddagger (x_\alpha')}{(1+\eps_1+\eps_2)} \leq \eps_1 \big(C^0_{\eps_1} + C_{\eps_1, \varphi}\big) + \eps_2 \big(C^0_{\eps_2} + C_{\eps_2, \varphi}\big).
    \end{equation}

    Then, for any compact set $K \subseteq E$ and $\eps_1 + \eps_2 \in (0,1)$, we have
    \begin{equation}\label{eq:main_step:reduction_to_boundary}
        \sup_{x \in K} u(x) - v(x) \leq \eps_1 \widetilde{C}_{1}^K +  \eps_2 \widetilde{C}_{2}^K + \sup_{x \in \widetilde{K}} u(x) - v(x),
    \end{equation}
    where $\widetilde{K} = \widehat{K} \cap \partial E$, $\widehat{K}$ as in \eqref{eqn:definition_hatK}, and $\widetilde{C}^K = (\widetilde{C}_{1}^K, \widetilde{C}_{2}^K)$ with
    \begin{equation}
        \widetilde{C}_{1}^K=\widetilde{C}_{1}^K(u-V_1,v-V_1) \quad\text{and}\quad \widetilde{C}_{2}^K=\widetilde{C}_{2}^K(u-V_2,v-V_2).
    \end{equation}

    In particular, the strict comparison principle holds for \eqref{eq:main_step:subsol} and \eqref{eq:main_step:supersol}.
\end{proposition}

\begin{proof}
    Let $u$ be a subsolution of $-H_+ f = 0$ and $v$ a supersolution of $-H_- f = 0$. For $\eps_1 + \eps_2 \in (0,1)$, $\alpha_1, \alpha_2 >1$, and $\varphi \in (0,1)$, consider the constructions in Propositions \ref{prop:optimizer_construction} and \ref{proposition:test_function_construction} for the subsolution $u$ and supersolution $v$.

    By Lemma \ref{lemma:test_functions_in_domain}, we have $f_\dagger \in \cD(H_+)$ and $f_\ddagger \in \cD(H_-)$ and, by Proposition \ref{proposition:test_function_construction} \ref{item:test_function_construction:fdagger_f_ddagger_test_functions_uv}, we find that $(x_{\alpha},x_{\alpha}')$ are the unique optimizers in
        \begin{equation} \label{eqn:proof_comparison_optimizinguf1vf2}
            \begin{aligned}
                u(x_{\alpha}) - f_\dagger(x_{\alpha}) & = \ssup{u - f_\dagger }, \\
                v(x_{\alpha}') - f_\ddagger(x_{\alpha}') & = \iinf{v - f_\ddagger },
            \end{aligned}
        \end{equation}
        which, by the sub- and supersolution properties for $H_+$ and $H_-$, respectively, implies that
        \begin{equation} \label{eqn:comparison_proof_subsuperestimate}
            0 \leq H_+ f_\dagger(x_{\alpha}) - H_- f_\ddagger(x_{\alpha}').
        \end{equation}
    
    We now work towards a contradiction and assume that, for every $\eps_1 + \eps_2 \in (0,1)$, $\alpha_1, \alpha_2 >1$, and $\varphi \in (0,1)$, we have $x_\alpha, x'_\alpha \in \interior{E}$.
    However, using \eqref{eqn:comparison_proof_subsuperestimate} and the assumptions above, we find
    \begin{align*}
        0 &\leq \liminf_{\alpha \rightarrow \infty} H_+ f_\dagger(x_{\alpha}) - H_- f_\ddagger(x_{\alpha}')= \liminf_{\alpha \rightarrow \infty} \bH f_\dagger(x_{\alpha}) - \bH f_\ddagger(x_{\alpha}')\\
        &\leq \eps_1 \big(C^0_{\eps_1} + C_{\eps_1, \varphi}\big) + \eps_2 \big(C^0_{\eps_2} + C_{\eps_2, \varphi}\big).
    \end{align*}
    To show that the right-hand side is strictly negative for a suitable choice of the parameters, thereby contradicting the assumption that both optimizing points lie in $\interior{E}$, we distinguish whether $E_2$ has a boundary, as this determines the required signs of the constants $C^0_{\eps_1}$ and $C^0_{\eps_2}$.\\
    \textit{Case $\partial E_2 = \emptyset$:} Note that we have $C^0_{\eps_1}<0\leq C^0_{\eps_2}$. Thus, letting $\varphi \downarrow 0$ we get a contradiction for some $\eps_1$ and small $\eps_2$, as $\lim_{\varphi \downarrow 0} C_{\eps, \varphi} = 0$. Consequently, for small $\varphi$ and $\eps$ and large $\alpha$, we have $x_\alpha \in \partial E_1 \times E_2 = \partial E$ or $x_\alpha' \in \partial E$.\\
    \textit{Case $\partial E_2 \neq \emptyset$:} Note that we have $C^0_{\eps_1}, C^0_{\eps_2}<0$. Thus, we get a contradiction as $\varphi \downarrow 0$. Consequently, for any $\eps_1 + \eps_2 \in (0,1)$ and small $\varphi \in (0,1)$, we have $x_\alpha \in \big(\partial E_1 \times E_2\big) \cup \big(E_1 \times \partial E_2\big) = \partial E$ or $x_\alpha' \in \partial E$. \\
    In any case, we can conclude that any limiting point lies on the boundary, i.e., for any $z$ as in Proposition \ref{prop:optimizer_construction} \ref{item:proposition:optimizing_limits}, we have $z \in \partial E$.

    Proceeding from the setup in Proposition \ref{prop:optimizer_construction} \ref{item:proposition:optimizing_point_construction_estimate_u-v} and \ref{item:proposition:optimizing_limits}, but now using that any limit point $z$ lies in 
    \begin{multline*}
        \widetilde{K} \coloneqq \left\{z \in \partial E \, \middle| \,  \frac{1}{1\boxminus\eps} \left(\eps \bfV(z) - u(z)\right) + \frac{1}{1\boxplus\eps} \left(\eps \bfV(z) - v(z) \right)   \right. \\
        \leq \left. \frac{\eps}{1\boxminus\eps} \ssup{\bfV-u}_K + \frac{\eps}{1\boxplus\eps} \ssup{\bfV-v}_K - \ssup{u-v}_K \right\},
    \end{multline*}
    we find
    \begin{align}
         &\ssup{u-v}_K
         \leq \lim_{\varphi \downarrow 0} \lim_{\alpha \rightarrow \infty} \frac{1}{1\boxminus\varepsilon}\big(P^\alpha[u](y_\alpha)  - \varepsilon(1-\varphi) \bfV(y_\alpha)\big)\\
         &\qquad - \frac{1}{1\boxplus\varepsilon}\big(P_\alpha[v](y_\alpha') + \varepsilon (1-\varphi) \bfV(y_\alpha')\big)\\
         &\qquad + \varphi \frac{2\varepsilon}{(1\boxminus\eps)(1\boxplus\eps)} \frac{1}{\alpha} + \frac{\varepsilon}{1\boxminus\varepsilon}c_{K,\alpha}^u + \frac{\varepsilon}{1\boxplus\varepsilon}c_{K,\alpha}^v\\
         & \quad \leq \frac{1}{1\boxminus\varepsilon}\left(u(z) - \varepsilon \bfV(z)\right) - \frac{1}{1\boxplus\varepsilon}\left(v(z) + \varepsilon \bfV(z)\right) + \frac{\varepsilon}{1\boxminus\varepsilon}c_K^u + \frac{\varepsilon}{1\boxplus\varepsilon}c_K^v\\
         & \quad \leq u(z) - v(z) + \varepsilon_1 \widehat{C}_1^K + \varepsilon_2 \widehat{C}_2^K + \frac{\varepsilon}{1\boxminus\varepsilon}c_K^u + \frac{\varepsilon}{1\boxplus\varepsilon}c_K^v, \label{eq:u-v_with_compact_knowledge}
    \end{align}
    where the last inequality is due to the fact that
    \begin{align}
        \frac{1}{1\boxminus\varepsilon}\left(u(z) - \varepsilon \bfV(z)\right) &= u(z) + \frac{\eps_1}{1\boxminus\varepsilon} \left(u(z) - V_1(z)\right) + \frac{\eps_2}{1\boxminus\varepsilon} \left(u(z) - V_2(z)\right),\\
        - \frac{1}{1\boxplus\varepsilon}\left(v(z) + \varepsilon \bfV(z)\right) &= -v(z) + \frac{\eps_1}{1\boxplus\varepsilon} \left(v(z) - V_1(z)\right) + \frac{\eps_2}{1\boxplus\varepsilon} \left(v(z) - V_2(z)\right),
    \end{align}
    $E_1$ is bounded and, depending on the case, $E_2$ is bounded or, if $E_2$ is unbounded, we have $u, v \in o(V_2)$ and thus $(u - \bfV)$ and $(v-\bfV)$ are bounded from above. In any case, we find constants $\widehat{C}^K = (\widehat{C}_1^K, \widehat{C}_2^K)$ such that the last inequality in \eqref{eq:u-v_with_compact_knowledge} holds.

    Consequently, we have shown that, for any compact set $K \subseteq E$, we have
    \begin{equation}
        \ssup{u-v}_K \leq \eps_1 \widetilde{C}_1^K + \eps_2 \widetilde{C}_2^K + \ssup{u-v}_{\widetilde{K}}
    \end{equation}
    where $\widetilde{C}^K = (\widetilde{C}_1^K, \widetilde{C}_2^K)$ is defined via the last inequality of \eqref{eq:u-v_with_compact_knowledge}.

\end{proof}

\begin{remark}
    Note that the constants $\widetilde{C}^K = (\widetilde{C}_1^K, \widetilde{C}_2^K)$ in the above proof depend on the initially chosen compact set $K$ but are independent of $\eps$.
\end{remark}

\subsection{Proof of Theorem \ref{th:comparison_HJI}}\label{sec:proof_main_thm}

\begin{lemma}\label{lemma:decomposition_A,B}
     Let $\bA$ and $\bB$ satisfy Assumption \ref{assumption:domain_setup} and Assumption \ref{assumption:compatibility} \ref{item:compatA} and \ref{item:compatB}, respectively. Fix $z_0,z_1\in \bR^q$ and $p \in \bR^q$. Let $\Xi = \Xi_{z_0,p,z_1}$ as in Definition \ref{definition:perturbation_first_second_order}, $\bfV$ a Lyapunov function as in Assumption \ref{item:assumption_Lyapunovfunction} in Theorem \ref{th:comparison_HJI}, and, for $\widehat{f}_1, \widehat{f}_2 \in C_c^\infty (\domain)$, $\eps = (\eps_1, \eps_2)$ with $\eps_1 + \eps_2 \in (0,1)$, and $\varphi \in (0,1)$, set
    \begin{align*}
        \widehat{f}_\dagger &=  (1\boxminus\eps) \widehat{f}_1 + (1-\varphi)\eps \bfV + \varphi\eps \Xi, \\
        \widehat{f}_\ddagger &=  (1\boxplus\eps) \widehat{f}_2 - (1-\varphi)\eps\bfV - \varphi\eps\Xi.
    \end{align*}
    For $z \in \domain$, set $f_\dagger = \widehat{f}_\dagger \circ s_z$, and $f_\ddagger = \widehat{f}_\ddagger \circ s_z$. Then, for $A_+$, $A_-$, $B_+$, and $B_-$ defined, analogously to $H_+$ and $H_-$, cf.\ Assumption \ref{assumption:domain_setup}, the following statements hold:
      \begin{enumerate}[(a)]
          \item $f_\dagger \in \cD(A_+)$ and $f_\ddagger \in \cD(A_-)$. Suppose furthermore that $\bA$ is linear on its domain, then
              \begin{align}\label{eq:lemma:linearity_testfunction}
                \frac{A_+ f_\dagger}{1\boxminus\varepsilon} & = \bA(\widehat{f} \circ s_z) + \frac{\varepsilon}{1\boxminus\varepsilon} (1-\varphi) A_+ \left(\bfV \circ s_z\right) + \frac{\varepsilon}{1\boxminus\varepsilon} \varphi \bA \left(\Xi \circ s_z\right),\\
                \frac{A_- f_\ddagger}{1\boxplus\varepsilon} & = \bA(\widehat{f}\circ s_z) - \frac{\varepsilon}{1\boxplus\varepsilon} (1-\varphi) A_+ \left(\bfV \circ s_z\right) - \frac{\varepsilon}{1\boxplus\varepsilon} \varphi \bA \left(\Xi \circ s_z\right).
            \end{align}
            \item $f_\dagger, \widehat{f}_\dagger \in \cD(B_+)$ and $f_\ddagger, \widehat{f}_\ddagger \in \cD(B_-)$. Suppose furthermore that $\bB$ is convex semi-monotone, then for any $x,y$ such that $z = x-y$, we have
            \begin{align}\label{eq:lemma:convexity_testfunction}
                \frac{B_+ f_\dagger}{1\boxminus\varepsilon}(x) & \leq \frac{1}{1\boxminus\varepsilon} \left(B_+ f_\dagger (x) - B_+ \widehat{f}_\dagger (y) \right)  + \bB \widehat{f} (y) \\
                & \hspace{2cm} +  \frac{\varepsilon}{1\boxminus\varepsilon} (1-\varphi) B_+ \bfV(y) + \frac{\varepsilon}{1\boxminus\varepsilon} \varphi B_+ \Xi(y),\\
                \frac{B_- f_\ddagger}{1\boxplus\varepsilon}(x) & \geq \frac{1}{1\boxplus\varepsilon}\left(B_- f_\ddagger(x) - B_- \widehat{f}_\ddagger(y) \right) + \bB \widehat{f}(y) \\
                & \hspace{2cm} - \frac{\varepsilon}{1\boxplus\varepsilon} (1-\varphi) B_- \bfV(y) - \frac{\varepsilon}{1\boxplus\varepsilon} \varphi B_- \Xi(y).
            \end{align}
      \end{enumerate}
\end{lemma}

\begin{proof}
The domain statements $f_\dagger \in \cD(A_+)$, $f_\ddagger \in \cD(A_-)$, $f_\dagger,\widehat{f}_\dagger \in \cD(B_+)$ and $f_\ddagger,\widehat{f}_\ddagger \in \cD(B_-)$ follow by Lemma \ref{lemma:test_functions_in_domain}. The two statements in \eqref{eq:lemma:linearity_testfunction} follow from the linearity of $A_+$. The two statements in \eqref{eq:lemma:convexity_testfunction} follow from the convex semi-monotonicity of $B_+$. We show only the first statement as the second follows analogously. First note that since $1\boxminus\eps = 1-\eps_1-\eps_2$ and $1\boxplus \eps = 1+\eps_1+\eps_2$, for any $y \in E$, we have
\begin{align}
    B_+ \widehat{f}_\dagger(y) &\leq (1-\eps_1-\eps_2) B_+ \widehat{f}(y) + \eps_1 (1-\varphi) B_+ V_1(y) + \eps_2 (1-\varphi) B_+ V_2(y)\\
    &\qquad + \eps_1 \varphi B_+ \Xi_1(y) + \eps_2 \varphi B_+ \Xi_2(y)\\
    & = (1\boxminus\eps) B_+ \widehat{f}(y) + \eps (1-\varphi) B_+ \bfV(y) + \eps \varphi B_+ \Xi(y)
\end{align}
where we used the convex semi-monotonicity of $B_+$.\ 
Now, for any $x \in E$ such that $z = x-y$, subtracting $B_+ \widehat{f}_\dagger(y)$, adding $B_+f_\dagger(x)$, and then dividing by $1\boxminus\eps$ on both sides yields the claim.
\end{proof}

\begin{proof}[Proof of Theorem \ref{th:comparison_HJI}]\label{proof:main_thm}
To prove the inequality in equation \eqref{eq:th:final}, and consequently the strong comparison principle, it suffices by Proposition \ref{proposition:basic_comparison_using_Hestimate} to establish \eqref{eq:main_step:H-H_estimate}, which is a condition on the interior operator. 
By the compactness of $\Theta$ and the lower semi-continuity of $\cI$ in $\theta$ assumed in \ref{item:assumption_directHJI_lsc}, we can find a $\theta^*_{\alpha} \in \Theta$ such that
\begin{align}
    \bH f_\dagger(x_\alpha) &= \sup_{\theta \in\Theta}\left\{\bA_{\theta} f_\dagger(x_\alpha) + \bB_{\theta} f_\dagger(x_\alpha) - \cI(x_\alpha,\theta)\right\}\\
    &=\bA_{\theta^*_{\alpha}} f_\dagger(x_\alpha) + \bB_{\theta^*_{\alpha}} f_\dagger(x_\alpha) - \cI(x_\alpha,\theta^*_{\alpha}).
\end{align}
Consequently, we can estimate
\begin{align}
    \frac{1}{1\boxminus\varepsilon}  \bH f_\dagger(x_\alpha) - \frac{1}{1\boxplus\varepsilon} \bH f_\ddagger(x'_\alpha)  \leq &\left[\frac{1}{1\boxminus\varepsilon} \bA_{\theta^*_{\alpha}} f_\dagger(x_\alpha) - \frac{1}{1\boxplus\varepsilon}\bA_{\theta^*_{\alpha}}f_\ddagger(x'_\alpha)\right]\\
    + &\left[\frac{1}{1\boxminus\varepsilon}  \bB_{\theta^*_{\alpha}} f_\dagger(x_\alpha) - \frac{1}{1\boxplus\varepsilon} \bB_{\theta^*_{\alpha}} f_\ddagger(x'_\alpha)\right]\\
    + &\left[ \frac{1}{1\boxplus\varepsilon}  \cI(x'_\alpha,\theta^*_{\alpha}) - \frac{1}{1\boxminus\varepsilon}\cI(x_\alpha,\theta^*_{\alpha})\right].
\end{align}

Using the expansions of $A_+$ and $A_-$ in Lemma \ref{lemma:decomposition_A,B} and the existence of a $\Phi$-controlled growth coupling, cf.\ Assumption \ref{item:assumption:Acoupling} of Theorem \ref{th:comparison_HJI}, we can make analogous arguments to \cite[Proof of Theorem 3.1, Estimate (1)]{DCFuKrNe24} for the difference of $\bA$.

For the difference of $\bB$, we can use the expansions of $B_+$ and $B_-$ in Lemma \ref{lemma:decomposition_A,B} and convex semi-monotonicity, cf.\ Assumption \ref{item:assumption:Bmonotone} of Theorem \ref{th:comparison_HJI}, and then proceed analogously to \cite[Proof of Theorem 3.1, Estimate (2)]{DCFuKrNe24}.

Lastly, we use Assumption \ref{item:assumption_directHJI_lsc} of Theorem \ref{th:comparison_HJI} to estimate as in \cite[Proof of Theorem 3.1, Estimate (3)]{DCFuKrNe24}.

Taking everything together, we arrive at the estimate
\begin{align}
    &\liminf_{\alpha_1, \alpha_2 \rightarrow \infty} \frac{\bH f_\dagger (x_\alpha)}{1\boxminus\eps} - \frac{\bH f_\ddagger (x_\alpha')}{1\boxplus\eps}\\ 
    &\quad\leq \frac{2\eps}{(1\boxminus\eps)(1\boxplus\eps)} \big( (1-\varphi) A_{\theta^*, +} \bfV(z) + \varphi \bA_{\theta^*}(\Xi_{z,0,z})(z) \big) \\
    &\quad   \quad + \frac{2\eps}{(1\boxminus\eps)(1\boxplus\eps)} \big( (1-\varphi) B_{\theta^*, +} \bfV(z) + \varphi \bB_{\theta^*}(\Xi_{z,0,z})(z)\big) \\
    & \quad  \quad - \frac{2\varepsilon}{1-\varepsilon^2} (1-\varphi) \cI(z,\theta^*) \\
    &\quad  \leq \frac{2\eps}{(1\boxminus\eps)(1\boxplus\eps)} (1-\varphi)\ssup{(A_{\theta^*, +} + B_{\theta^*, +})(\bfV) - \cI(\cdot, \theta^*)}_{\widehat{K}} \\
    & \quad\quad + \frac{2\eps}{(1\boxminus\eps)(1\boxplus\eps)} \varphi \ssup{(\bA_{\theta^*} + \bB_{\theta^*})(\Xi_{\cdot,0,\cdot})}_{\widehat{K}}  \\
    &\quad \leq \eps \bigg( \frac{2}{(1\boxminus\eps)(1\boxplus\eps)} c_{\bfV} + \frac{2}{(1\boxminus\eps)(1\boxplus\eps)} \varphi \ssup{(\bA_{\theta^*} + \bB_{\theta^*})(\Xi_{\cdot,0,\cdot})}_{\widehat{K}} \bigg)  \\
    &\quad = \eps \left(C_\eps^0 + C_{\eps,\varphi} \right),
\end{align}
where $c_{\bfV} = c'_{V_1} + c'_{V_2}$. If $\partial E_i \neq \emptyset$, we set $c'_{V_i} \coloneqq c_{V_i}$ and, if $\partial E_i = \emptyset$, we set $c_{V_i} \coloneqq (c_{V_i} \vee 0)$ with $c_{V_i}$ given by \eqref{eq:mainth_HJI_bound}. Furthermore, let $C_\eps^0 = (C_{\eps,1}^0, C_{\eps, 2}^0)$ and $C_{\eps,\varphi} = (C_{\eps,\varphi,1}, C_{\eps,\varphi,2})$ defined via the second to last line. The estimate on the difference of Hamiltonians \eqref{eq:main_step:H-H_estimate} of Proposition \ref{proposition:basic_comparison_using_Hestimate} is thus satisfied. As a consequence, the final estimate \eqref{eq:th:final} holds and the strict comparison principle follows.
\end{proof}

\appendix
\section{The Jensen perturbation}\label{sec:jensen}

This section contains the Jensen perturbation results needed for the test function construction.
Note that the notion of semi-convexity used here is defined on open sets. Thus, it is crucial to consider an open neighborhood of $E_i$, if it is a (partially) bounded set.

\begin{proposition} \label{proposition:Jensen_Alexandrov_cutoff}
    Fix $\eta > 0$ and let $i \in \{1,2\}$.\ Let $\phi \colon \cN(E_i \times E_i) \rightarrow \bR$ be bounded above and semi-convex with convexity constant $\kappa_i$ on $\cN(E_i \times E_i)$. Suppose that $(x_0,y_0) \in E_i \times E_i$ is an optimizer of
    \begin{equation*}
        \phi(x_0,y_0) = \ssup{\phi}.
    \end{equation*}
    Let $R_i > 0$, $\{\zeta_{i,z,p}\}_{z \in E_i, p \in \bR^{q_i}} \subseteq C(E_i)$ and $\{\xi_{i,z}\}_{z \in E} \subseteq C^1(E_i)$ and semi-concavity constant $\kappa_{\xi_i}$ be as in Definition \ref{definition:perturbation_first_second_order}.

    Fix $\epsilon_1,\epsilon_2 > 0$ such that $1 - (\epsilon_1+\epsilon_2) \kappa_{\xi_i} > 0$. Furthermore, define for $p= (p_1,p_2) \in \bR^{q_i} \times \bR^{q_i}$ the perturbed functions
    \begin{equation} \label{eqn:def:phi_p_eta}
        \phi_{p}(x,y) \coloneqq \phi(x,y) - \epsilon_1 \big(\xi_{i,x_0} (x) + \zeta_{i,x_0, p_1} (x)\big) - \epsilon_2 \big( \xi_{i,y_0} (y) + \zeta_{i,y_0, p_2} (y)\big).
    \end{equation}
    Then, there exist $p_1 \in B_\eta(0)\cap \{E_i - x_0\}$, $p_2 \in B_\eta(0)\cap \{E_i - y_0\}$, and a pair $(x_1,y_1) \in (B_\eta(x_0) \cap E_i) \times (B_\eta(y_0) \cap E_i)$ globally maximizing $\phi_{p}$ at which $\phi_{p}$ is twice differentiable. 
\end{proposition}

\begin{proof}
    The statement follows from \cite[Proposition A.1]{DCFuKrNe24} by noting that choosing $B_\eta(0)\cap \{E_i - x_0\}$ and $B_\eta(0)\cap \{E_i - y_0\}$, respectively, results in the same properties of the set-valued map $\text{Opt}$.
\end{proof}

\begin{corollary} \label{corollary:Jensen_optimizerbound}
    For $\eta >0$, $p$ and $(x_1,y_1)$ as in Proposition \ref{proposition:Jensen_Alexandrov_cutoff}, we have

    \begin{equation}\label{eqn:Jensen_control_optimizationproblem_onlyPerturb}
        0 \leq - \epsilon_1 \big(\xi_{x_0} (x_1) + \zeta_{x_0, p_1} (x_1)\big) - \epsilon_2 \big( \xi_{y_0} (y_1) + \zeta_{y_0, p_2} (y_1)\big) \leq \epsilon_1 \eta + \epsilon_2 \eta,
    \end{equation}
    and
    \begin{equation} \label{eqn:Jensen_control_optimizationproblem}
        \ssup{\phi} \leq \phi_{p,\epsilon}(x_1,y_1) \leq \ssup{\phi} + \epsilon_1 \eta + \epsilon_2 \eta.
    \end{equation}
\end{corollary}
\begin{proof}
    See \cite[Corollary A.2]{DCFuKrNe24}.
\end{proof}

\printbibliography
 
\end{document}